\documentclass[final,leqno]{amsart}

\usepackage[dvips]{graphicx}
\usepackage{amsmath,amsfonts,amssymb,latexsym}
\usepackage{graphicx,float,epsfig}
\usepackage[dvips]{color}

\newtheorem{remark}{\textit{Remark}}[section]
\newtheorem{problem}{Problem}[section]
\newtheorem{lemma}{Lemma}[section]
\newtheorem{proposition}{Proposition}[section]
\newtheorem{theorem}{Theorem}[section]
\newtheorem{corollary}{Corollary}[section]

\def\bA{\mathbf{A}}

\def\bv{\mathbf{v}}
\def\bx{\mathbf{x}}

\def\bp{\boldsymbol{p}}

\def\bthe{\boldsymbol{\theta}}
\def\beth{\boldsymbol{\eta}}
\def\bta{\boldsymbol{\tau}}
\def\bsi{\boldsymbol{\sigma}}
\def\bt{\boldsymbol{t}}
\def\bfX{\boldsymbol{X}}
\def\bfV{\mathbf{V}}
\def\bVE{\bfV_h^E}
\def\bWE{W_h^E}

\def\Cten{\boldsymbol{\mathcal{C}}}
\def\CE{\mathcal{E}}

\def\CT{\mathcal{T}}

\def\dim{\mathop{\mathrm{\,dim}}\nolimits}
\def\disp{\displaystyle}
\def\Div{\mathop{\mathbf{div}}\nolimits}
\def\div{\mathop{\mathrm{div}}\nolimits}

\def\ga{\boldsymbol{\gamma}}
\def\G{\Gamma}

\def\hCT{\widehat{\CT}}

\def\HdO{{H^2(\O)}}

\def\HuE{H^1(E)}
\def\Hu2E{H^2(E)}

\def\HsE{H^{s}(E)}
\def\HuO{{H^1(\O)}}

\def\HsO{{H^{s}(\O)}}
\def\I{\mathbf{I}}

\def\l{\lambda}
\def\LO{{L^2(\O)}}
\def\LE{{L^2(E)}}
\def\O{\Omega}
\def\PiO{\Pi_{0}^{E}}
\def\PiE{\Pi_{\varepsilon}^E}

\def\R{{\mathbb{R}}}
\def\rot{\mathop{\mathrm{rot}}\nolimits}
\def\Rot{\mathop{\mathbf{rot}}\nolimits}

\def\SZ{\mathrm{c}}
\def\tr{\mathop{\mathrm{tr}}\nolimits}

\newcommand{\norm}[1]{\left\|#1\right\|}
\newcommand\0{\boldsymbol{0}}
\newcommand\bn{\boldsymbol{n}}
\newcommand\bbP{\mathbb{P}}

\def\vv{{\mathsf v}}

\begin{document}

%-------	Title and Author	------------------------------

%\title{A Virtual Element Method for Reissner-Mindlin plates}

\title{Virtual Elements for a shear-deflection formulation of Reissner-Mindlin plates}

\author{L. Beir\~ao da Veiga}
\address{Dipartimento di Matematica e Applicazioni,
Universit\`a di Milano-Bicocca,
20125 Milano, Italy.}
\email{lourenco.beirao@unimib.it}
\author{D. Mora}
\address{Departamento de Matem\'atica, Universidad del B\'io-B\'io,
Casilla 5-C, Concepci\'on, Chile and CI$^2$MA, Universidad de Concepci\'on, Concepci\'on, Chile.}
\email{dmora@ubiobio.cl}
\thanks{The second author was partially supported by CONICYT-Chile
through FONDECYT project 1140791 (Chile) and by DIUBB through project 151408 GI/VC,
Universidad del B\'io-B\'io, (Chile)}
\author{G. Rivera}
\address{Departamento de Ciencias Exactas,
Universidad de Los Lagos, Casilla 933, Osorno, Chile.}
\email{gonzalo.rivera@ulagos.cl}
\thanks{The third author was supported by a CONICYT fellowship (Chile).}

\subjclass[2000]{Primary 65N30, 65N12, 74K20, 74S05, 65N15.}

\keywords{Virtual element method, Reissner--Mindlin plates, error analysis, polygonal meshes.}

%\begin{frontmatter}

\begin{abstract} 
We present a virtual element method for the Reissner--Mindlin plate bending
problem which uses shear strain and deflection as discrete variables
without the need of any reduction operator. The proposed method is
conforming in $[H^{1}(\O)]^2 \times H^2(\O)$ and has the advantages 
of using general polygonal meshes and yielding
a direct approximation of the shear strains.
The rotations are then obtained by a simple postprocess
from the shear strain and deflection.
We prove convergence estimates with involved constants that
are uniform in the thickness $t$ of the plate.
Finally, we report numerical experiments which allow us to assess
the performance of the method.
\end{abstract}

\maketitle

%\pagestyle{myheadings}
%\thispagestyle{plain}
%\markboth{L. BEIR\~AO DA VEIGA, D. MORA, AND G. RIVERA}
%{A VIRTUAL ELEMENT METHOD FOR REISSNER-MINDLIN PLATES}

%\date{\today}
%\end{frontmatter}

%-----------------------------------------------------------------------

%\setcounter{equation}{0}
\section{Introduction}
\label{SEC:INTR}

The {\it Virtual Element Method} (VEM), introduced in \cite{BBCMMR2013,BBMR2014},
is a recent generalization of the Finite Element Method
which is characterized by the capability of dealing with
very general polygonal/polyhedral meshes.
The interest in numerical methods that can make use of general
polytopal meshes has recently undergone a significant growth
in the mathematical and engineering literature; among
the large number of papers on this subject, we cite as a minimal sample
\cite{AHSV,BBCMMR2013,BLMbook2014,CGH14,DPECMAME2015,RW,ST04,TPPM10}.

%{\bf Please in the previous list cite also one paper by "Weisser", he did stuff on polygonal elements}

Indeed, polytopal meshes can be very useful for
a wide range of reasons, including meshing of
the domain (such as cracks) and data (such as inclusions) features,
automatic use of hanging nodes, use of moving meshes, adaptivity.
Moreover, the VEM presents the advantage to easily implement
highly regular discrete spaces.
Indeed, by avoiding the explicit construction of the local basis
functions, the VEM can easily handle general polygons/polyhedrons
without complex integrations on the element
(see \cite{BBMR2014} for details on the coding aspects of the method).
The Virtual Element Method has been applied successfully in a
large range of problems, see for instance
\cite{AABMR13,ABMVsinum14,ALM15,BBCMMR2013,BBMR2014,BLM2015,BMRR,BBBPS2016,BBPS2014,ultimo,BM12,CG2016,ChM-camwa,Paulino-VEM,MRR2015,PG2015,PPR15,WRR,ZChZ2016}.

The Reissner--Mindlin plate bending
problem is used to approximate the deformation of a thin
or moderately thick elastic plate. Nowadays, it is very well understood that
the discretization of this problem poses difficulties
due to the so called locking phenomenon when the thickness $t$ is
small with respect to the other dimensions of the plate.
Nevertheless, adopting for instance a
reduced integration or a mixed interpolation technique, this phenomenon can be avoided. Indeed,
several families of methods have been rigorously shown to be free from locking and optimally
convergent. We mention \cite{falk,Lova} for a thorough description and further references.

%{\bf Please in addition to Falk (here above) cite also the paper that you find at the following
%link, its a nice review by Carlo Lovadina: \\
%http://www.imse08.unican.es/files/abstract-lovadina.pdf}

Recently, a new approach to solve the Reissner--Mindlin
bending problem has been presented in~\cite{BHKLNRS} by Beir\~ao da Veiga et al.
(see also \cite{EOB2013,LBC2012}).
In this case a variational formulation of the plate bending problem is written in terms of
shear strain and deflection with the advantage that the ``shear locking phenomenon''
is avoided. A discretization of the problem by Isogeometric Analysis
is proposed. Under some regularity assumptions on the exact solution,
optimal error estimates with constants independent of the plate thickness are proved.

The aim of this paper is on developing a Virtual Element Method
which applies to general polygonal (even non-convex) meshes
for Reissner-Mindlin plates. 
We consider a variational formulation written in
terms of shear strain and deflection presented in \cite{BHKLNRS}.
Here, we exploit the capability of VEM to built highly regular discrete spaces
and propose a conforming $[H^{1}(\O)]^2 \times H^2(\O)$ discrete formulation,
respectively for the shear strain and deflections.
The resulting bilinear form is continuous and elliptic
with appropriate $t$-dependent norms.
This method makes use of a very simple set of degrees of freedom,
namely 5 degrees of freedom per vertex of the mesh plus the number of edges,
and approximates directly the transverse shear strain, which is distinctive of this approach.
Moreover, the rotations are obtained by a simple postprocess
from the shear strain and deflection.
Under some regularity assumptions on the exact solution,
optimal error estimates (in the natural norms of the adopted formulation)
with constants independent of the plate
thickness are proved for all the involved variables. 
In addition, we present error estimates in weaker norms
using a duality argument.
Furthermore, let us remark that
it is possible to generalize the proposed scheme to a family
of high order methods, by considereing the $C^1(\O)$
family of elements in \cite{BM12} and combining it with
a VEM rotation space of higher degree.
Finally, we point out that, differently
from the finite element method where building globally $C^1(\O)$
functions is complicated, here the virtual deflection
space can be built with a rather simple construction
due to the flexibility of the virtual approach.
Moreover, the present analysis constitutes a stepping
stone towards the more challenging goal of devising virtual element
approximations for other problems, as laminated or stiffened plates, or shells.
In a summary, the advantages of the proposed method are the possibility
to use general polygonal meshes and a better conformity with the limit
Kirchhoff problem, ensuing from the $H^2(\O)$
approximation used for the discrete deflection.

The outline of this article is as follows: we introduce in
Section~\ref{SEC:Model} the Reissner-Mindlin plate model,
first in terms of deflection and rotations variables and
then in an equivalent form in terms of deflection and
transverse shear strain variable. In Section~\ref{SEC:Discrete},
we present the  discrete spaces for the shear strain and deflection,
together with their properties,
next, we construct the discrete bilinear forms
and the loading term. We end this section with the
presentation of the virtual element discrete formulation.
In Section~\ref{SEC:approximation}, we present 
the error analysis of the virtual scheme.
In Section~\ref{SEC:NUMER}, we report a couple of numerical tests that
allow us to assess the convergence properties of the method.

Throughout the paper, $\O$ is a generic Lipschitz bounded domain of $\R^2$. For $s\geq 0$,
$\norm{\cdot}_{s,\O}$ stands indistinctly for the norm of the Hilbertian
Sobolev spaces $\HsO$ or $[\HsO]^2$ with the convention
$H^0(\O):=\LO$. Finally, we employ $\0$ to denote a generic null vector
and we will denote with $C$ a generic constant which may take
different values in different occurrences, and which is independent of the mesh
parameter $h$ and the plate thickness $t$. 

%{\bf PLEASE: do a deep check at all the bibliography, there are some papers
%(preprints, submitted, etc..) that can be updated!}

%\setcounter{equation}{0}
\section{Continuous problem}
\label{SEC:Model}
Consider an elastic plate of thickness $t$, $0<t\le1$,
with reference configuration $\O\times(-t/2,t/2)$,
where $\O$ is a convex polygonal domain of $\R^2$ occupied
by the mid-section of the plate. The deformation of the plate
is described by means of the Reissner-Mindlin model in terms of the
rotations $\boldsymbol{\theta}=(\theta_1,\theta_2)$ of the fibers initially normal
to the plate mid-surface and the deflection $w$.
We subdivide the boundary $\G$ of $\O$ in three 
disjoint parts such that,
$$\G=\G_c\cup \G_s\cup\G_f.$$
The plate is assumed to be clamped on $\G_c$, simply supported on $\G_s$ 
and  free on $\G_f$. We assume that $\G_c$ has positive measure.
We denote by $\bn$ the outward unit normal vector
to $\G$, the following equations describe the plate response
to a conveniently scaled transverse load $g$:
\begin{equation}\label{Reissner-M}
\left\{\begin{array}{ll}
-\boldsymbol{\div} \Cten\varepsilon(\bthe)-\l 
t^{-2}(\nabla w-\bthe) &=\0\quad\quad in\,\ \O, 
\\
-\div(\l t^{-2}(\nabla w-\bthe))&=g \quad\quad in\,\ 
\O,\\
\bthe=\0,\ \ w=0&\quad\quad\quad on\,\ \ \G_c,\\
\Cten\varepsilon(\bthe)\bn  =\0,\ \ 
w=0&\quad\quad\quad on\,\ \ \G_s,\\
\Cten\varepsilon(\bthe)\bn  
=\0,\ \ (\bthe-\nabla w)=\0&\quad\quad\quad on\,\ \ \G_f,
\end{array}\right.
\end{equation}
where $\l:= \mathbb{E}k/2(1 + \nu)$ is the shear modulus, with $\mathbb{E}$ being the Young modulus,
$\nu$ the Poisson ratio, and $k$ a correction factor,
$\varepsilon(\bthe):=\frac{1}{2}(\nabla\bthe+(\nabla\bthe)^{t})$ is
the standard strain tensor, and $\Cten$ is the tensor of bending moduli,
given by (for isotropic materials)
\begin{equation*}
\label{ctao}
\Cten{\bsi}:=\dfrac{\mathbb{E}}{12(1-\nu^2)}
\left((1-\nu)\bsi+\nu\tr(\bsi)\I  \right),\qquad\bsi\in[L^2(\O)]^{2\times2},
\end{equation*}
where $\tr(\bsi)$ is trace of $\bsi$ and $\I$ is the identity tensor.

Let us consider the space
\begin{equation*}
% \label{continue-space}
 \widetilde{\bfX}:=\{(v,\beth)\in \HuO\times [\HuO]^2:
 v=0 \text{ on }\G_c\cup\G_s,\beth=\0\text{ on }\G_c\}.
\end{equation*}
By testing the system~\eqref{Reissner-M} with 
$(v,\beth)\in\widetilde{\bfX}$, integrating by parts and using the boundary conditions,
we write the following variational formulation:
\begin{problem}
\label{P_1}
Given $g\in\LO$, find $(w,\bthe)\in \widetilde{\bfX}$ such that 
\begin{equation*}
%\label{forvar2}
a(\bthe,\beth)+ b(\bthe-\nabla w,\beth-\nabla 
v)=(g,v)_{0,\O}\qquad \forall(v,\beth)\in\widetilde{\bfX},
\end{equation*}
where $(\cdot,\cdot)_{0,\O}$ denotes the inner-product
in $\LO$, and the bilinear forms are given by
\begin{equation*}
\begin{split}
a(\bthe,\beth)&:=(\Cten\varepsilon(\bthe),\varepsilon(\beth))_{0,\O},\\
b(\bthe,\beth)&:=\l t^{-2}(\bthe,\beth)_{0,\O}.
\end{split}
\end{equation*}
\end{problem}

The following result states that
the bilinear form appearing in Problem~\ref{P_1} is coercive
(see \cite[Proposition A.1]{BHKLNRS}).
 \begin{lemma}
There exists a positive constant $\alpha$ depending
only on the material constants and the domain $\O$ 
such that:
\begin{equation}\label{cota_0}
 a(\beth,\beth)+  b(\beth-\nabla v,\beth-\nabla v)\geq \alpha\left( 
\|\beth\|^2_{1,\O}+t^{-2}\|\beth-\nabla v\|^2_{0,\O}+\| v\|^2_{1,\O}  
\right)\quad \forall(v,\beth)\in \widetilde{\bfX}.
\end{equation}
\end{lemma}
It is well known that the discretization of the Reissner-Mindlin equations
have difficulties due to the so called locking phenomenon
when the thickness $t$ is small with respect to the other dimensions of the plate.
To avoid this phenomenon we will introduce and analyze
an alternative formulation of the problem that 
does not suffer from such a drawback.
In order to simplify the notation, and without any 
loss of generality, we will assume $\l=1$ in the following.

\subsection{An equivalent variational formulation}

The variational formulation that will be considered here, was introduced in
the context of shells in \cite{EOB2013,LBC2012} and has been
studied in \cite{BHKLNRS} for Reissner-Mindlin plates using Isogeometric Analysis.

Now, we note that the equivalent formulation is derived
by simply considering the following change of variables:
\begin{equation}
\label{change}
 (w,\bthe)\longleftrightarrow (w,\boldsymbol{\gamma})\quad \text{ with } \quad
\bthe=\nabla w+\boldsymbol\gamma.
\end{equation}
We note that the physical 
interpretation of the variable $\boldsymbol\gamma$
corresponds to the transverse shear strain.

The equivalent formulation will be obtained by using the change of
the variables \eqref{change} in Problem~\ref{P_1}.

For the analysis we will consider the following
$t$-dependent energy norm:

\begin{equation}
 \label{norm}
 |||v,\bta|||^2:=\|\bta+\nabla 
v\|^2_{1,\O}+t^{-2}\|\bta\|^2_{0,\O}+\|v\|^2_{1,\O},
\end{equation}
for all sufficiently regular functions $\bta:\O\longrightarrow\R^2$
and $v:\O\longrightarrow\R$.

Now, we define the following variational spaces:
$$\widehat{\bfX}:=\overline{C^\infty(\O)\times[C^\infty(\O)]^2}^{
|||\cdot,\cdot|||};$$
$$\bfX:=\{(v,\bta)\in \widehat{\bfX}:v=0 \text{ on 
}\G_c\cup\G_s,\nabla v+\bta=\0\text{ on }\G_c  \}.$$
It is immediately verified that
$$\HdO\times [\HuO]^2\subset \widehat{\bfX}\subset \HuO\times 
[\LO]^2.$$
Moreover, note that the space $\bfX$ exactly corresponds to $\widetilde{\bfX}$ up to the 
change of variables \eqref{change}.

Let us introduce the equivalent variational formulation for the 
Reissner-Mindlin model as follows:
\begin{problem}
\label{P_2}
Given $g\in\LO$, find $(w,\ga)\in\bfX$ such that 
\begin{equation*}
\label{forvar.2}
a(\nabla w+\ga,\nabla v+\bta)+b(\ga,\bta)=(g,v)_{0,\O}\qquad\forall(v,\bta)\in \bfX.
\end{equation*}
\end{problem}

We have that Problem~\ref{P_2} is equivalent to Problem~\ref{P_1}
up to the change of variables \eqref{change}.
As a consequence, we have the following
coercivity property for the bilinear form on the left hand side
of Problem~\ref{P_2} (see \eqref{cota_0}):
\begin{equation}
\label{lax}
 a(\nabla v+\bta,\nabla v+\bta)+b(\bta,\bta)\geq \alpha 
|||v,\bta|||^2\quad \forall (v,\bta)\in \bfX,
\end{equation}
with same constant $\alpha$. Moreover, bilinear forms
$a(\cdot,\cdot)$ and $b(\cdot,\cdot)$ are bounded
uniformly in $t$.

Therefore, Problem~\ref{P_2} has a unique solution $(w,\ga)\in\bfX$ and
$$|||w,\ga|||\le C\Vert g\Vert_{0,\O}.$$

\section{Virtual element discretization}
\label{SEC:Discrete}

We begin this section, by recalling the mesh construction
and the shape regularity assumptions to introduce the discrete
virtual element spaces for the shear strain and deflection,
together with their properties, next, we will introduce
discrete bilinear forms and the loading term.
Finally, we end this section with the presentation of the
virtual element discretization of Problem~\ref{P_2}.

\subsection{Mesh regularity assumption}
\label{meshassup}

Let $\left\{\CT_h\right\}_h$ be a sequence of 
decompositions of $\O$ into polygons $E$.
Let $h_E$ denote the diameter of the element $E$ and
$h:=\disp\max_{E\in\CT_h}h_E$.

For the analysis, we will make the following
assumptions as in \cite{BBCMMR2013,BLV,BMRR}:
there exists a positive real number $C_{\CT}$ such that,
for every $h$ and every $E\in \CT_h$,
\begin{itemize}
\item[$\bA_1$:]  the ratio between the shortest edge
and the diameter $h_E$ of $E$ is larger than $C_{\CT}$;
\item[$\bA_2$:]  $E\in\CT_h$ is star-shaped with
respect to every point of a  ball
of radius $C_{\CT}h_E$.
\end{itemize}

For any subset $S\subseteq\R^2$ and nonnegative
integer $k$, we indicate by $\bbP_{k}(S)$ the space of
polynomials of degree up to $k$ defined on $S$.
To keep the notation simpler, we denote by $\bn$ a general normal
unit vector; in each case,  its precise definition will
be clear from the context and we denote by $\bt$ the tangent unit 
vector defined as the anticlockwise rotation of $\bn$.

To continue the construction of the discrete scheme,
we need some preliminary
definitions. First, we split the bilinear forms $a(\cdot,\cdot)$ and $b(\cdot,\cdot)$ introduced
in the previous section as follows:
\begin{eqnarray}
\label{formasa}
a\left(\nabla w+\ga, \nabla v+\bta \right)&=&\sum_{E\in\CT_h}a^E\left(\nabla w+\ga, \nabla 
v+\bta\right)\quad\forall (w,\ga), 
(v,\bta)\in \bfX,\\\label{formasb}
b\left(\ga,\bta\right)&=&\sum_{E\in\CT_h}b^E(\ga,\bta)\quad\qquad\forall \ga,\bta\in[\HuO]^2,
\end{eqnarray}
with
$$a^E\left(\nabla w+\ga, \nabla v+\bta 
\right):=\left(\Cten\varepsilon(\nabla w+ 
\bta),\varepsilon(\nabla v+\ga) \right)_{0,E}$$
and
$$b^E(\ga,\bta):=t^{-2}(\ga,\bta)_{0,E}.$$
Finally, we define
\begin{equation*}
\mathcal{A}((w,\ga), (v,\bta)):=a(\nabla w+\ga,\nabla 
v+\bta)+b(\ga,\bta)=\sum_{E\in\CT_h}\mathcal{A}^E((w,\ga), (v,\bta))
\quad \forall (w,\ga),(v,\bta)\in \bfX,
\end{equation*}
where
\begin{equation*}
\mathcal{A}^E((w,\ga),(v,\bta))
=a^E\left(\nabla w+\ga,\nabla v+\bta \right)+b^E\left(\ga,\bta\right).
\end{equation*}

In order to construct the discrete scheme associated to Problem~\ref{P_2},
in what follows, we will show that for each
$h>0$  it is possible to build the following:
\begin{enumerate}
 \item a discrete virtual space $\bfX_h\subseteq \bfX$ such that
 $$\bfX_h:=\{(v_h,\bta_h)\in (W_h\times\bfV_h):v_h=0 \text{ on 
}\G_c\cup\G_s,\nabla v_h+\bta_h=\0\text{ on }\G_c  \},$$
in which the virtual spaces 
$W_h\subseteq \HdO$ and $\bfV_h\subseteq [\HuO]^2 $;
\item a symmetric bilinear form $\mathcal{A}_h:\bfX_h\times\bfX_h\to\R$
which can be split as
\begin{equation}\label{localbili}
\mathcal{A}_h((w_h,\ga_h), (v_h,\bta_h)):=\sum_{E\in\CT_h}\mathcal{A}_h^E((w_h,\ga_h), (v_h,\bta_h))
\quad\forall(w_h,\ga_h),(v_h,\bta_h)\in \bfX_{h},
\end{equation}
with $\mathcal{A}_h^E(\cdot,\cdot)$ local bilinear forms on $\bfX_h|_{E} \times\bfX_h|_{E}$;
\item an element $g_h\in \bfX_h'$ and a discrete duality pair $\left<\cdot,\cdot\right>_h$
in such a way that the following  discrete problem:
Find $(w_h,\ga_h)\in \bfX_h$  such that 
\begin{equation}
\label{forvar3}
\mathcal{A}_h((w_h,\ga_h), (v_h,\bta_h))=\left<g_h,v_h\right>_h\qquad \forall 
(v_h,\bta_h)\in \bfX_h,
\end{equation}
admits  a unique solution $(w_h,\ga_h)\in \bfX_h$ and exhibits optimal approximation properties.
\end{enumerate}

\subsection{Discrete virtual spaces for shear strain and deflection}\label{sec:added:X}

We introduce a pair of finite dimensional spaces for shear strain and deflection:
$$\bfV_{h}\subseteq [\HuO]^{2},\qquad\qquad W_{h}\subseteq \HdO.$$

First, we construct the shear strain virtual space $\bfV_{h}$, inspired from \cite{ABMVsinum14}.
With this aim, we consider a simple polygon $E$ (meaning open simply connected
sets whose boundary is a non-intersecting line made of a finite number of straight line segments) and we define
\begin{equation*}
\label{espace-1}
\boldsymbol{\mathbb{B}}_{\partial E}:=\{\bta_{h}\in [C^0(\partial E)]^2: \bta_{h}\cdot 
\boldsymbol{t}|_{\partial E}\in 
\bbP_2(e)\text{ and }\bta_{h}\cdot 
\boldsymbol{n}|_{\partial E}\in 
\bbP_1(e)\ \ \forall e\in 
\partial E\}.
\end{equation*}
We then consider the finite dimensional space defined as follows:
\begin{equation*}
\label{space2}
\bVE:=\left\{\bta_{h}\in [\HuE]^2: \bta_{h}|_{\partial E}\in\boldsymbol{\mathbb{B}}_{\partial E},
\left\{\begin{array}{ll}
\disp -\Delta \bta_h  +\Rot s=\0 \text{ in } E,\\
\rot \bta_h\in\bbP_0(E),
\end{array}\right.
\text{ for some } s\in L^2(E)\right\}.
\end{equation*}
Note that the operators and equations above
are to be interpreted in the distributional sense.
The space $\bVE$ is well defined. Indeed, given
a (piecewise polynomial) boundary value
$\bta_{h}|_{\partial E}\in \boldsymbol{\mathbb{B}}_{\partial E}$,
the associated function $\bta_{h}$ inside the element $E$
is obtained by solving the Stokes-like variational problem
and using that
$$\rot \bta_{h}|_{E}=\dfrac{1}{|E|} \disp\int_{E}\rot \bta_{h}=\dfrac{1}{|E|}
\disp\int_{\partial E}\ \bta_{h}\cdot \boldsymbol{t}.$$
We observe that $\bta_{h}$ minimizes the $H^1(E)$-seminorm
over all the functions in $H^1(E)$ with constant $\rot$
and satisfying the fixed boundary condition on $\partial E$.

It is important to observe that, since the functions in $\bVE$
are uniquely identified by their boundary values, 
$\dim(\bVE)=\dim(\bVE|_{\partial E})$, i.e.,
$\dim(\bVE) = 3N_{E}$, with $N_E$ being the number of edges of $E$.
This leads to introducing
the following $3N_{E}$ degrees of freedom for the space $\bVE$:
\begin{itemize}
\item $\mathcal{V}_E^h$: the values of $\bta_h$ (vector) at the vertices of $E$.
\item $\CE_E^h$: the value of the 
\begin{equation*}
\disp \dfrac{1}{|e|}\int_{e}\bta_h\cdot\boldsymbol{t}\quad \forall \text{ 
edge }e\in \partial E.
\end{equation*}
\end{itemize}
Moreover, we note that as a consequence of the definition $\bVE$,
the output values of the two sets of degrees of freedom $\mathcal{V}_E^h$ and 
$\CE_E^h$ are sufficient to uniquely determine
$\bta_h\cdot\boldsymbol{t}$ and $\bta_h\cdot\boldsymbol{n}$ on the boundary of 
$E$, for any $\bta_h\in\bVE$. Finally, we note that clearly $[\bbP_1(E)]^2 \subset \bVE$.

For every decomposition $\CT_h$ of $\O$ into
simple polygons $E$, we define the global  space $\bfV_{h}$ without boundary conditions.
\begin{align*}
\label{uno}
\bfV_h:=\{\bta_{h}\in [\HuO]^2:\bta_h|_{E}\in \bVE\quad \forall E\in \CT_h\}.
\end{align*}
In agreement with the local choice of the degrees of freedom, in 
$\bfV_h$ we choose the following degrees of freedom:
\begin{itemize}
\item $\mathcal{V}^h$: the values of $\bta_h$ (vector) at the vertices of $\CT_h$.
\item $\CE^h$: the value of the 
\begin{equation*}
 \label{freedom2}
\disp \dfrac{1}{|e|}\int_{e}\bta_h\cdot\boldsymbol{t}\quad \forall
\text{ edge }e\in\CT_h.
\end{equation*}
\end{itemize}

Now, we will introduce the discrete virtual space $W_{h}$ for the deflection, see also \cite{BM12,ABMVsinum14}.
With this aim, we first define the following finite dimensional space:
\begin{equation*}
\begin{split}
\bWE:=\{v_{h}\in \Hu2E: \Delta^2v_{h}=0, v_{h}|_{\partial E}\in 
\bbP_3(e),\nabla v_{h}|_{\partial E}\in [C^0(\partial E)]^2\text{ and }
\partial_{\boldsymbol{n}}v_{h}|_{\partial E}\in \bbP_1(e)\,\,\forall e\in \partial E\},
\end{split}
\end{equation*}
where $\Delta^2$ represents the biharmonic operator. We observe
that any $v_{h}\in\bWE$ clearly satisfies the following conditions:
\begin{itemize}
\item the trace (and the trace of the gradient) on the boundary of $E$ is continuous; 
%\item the gradient on the boundary is continuous and on each edge its normal (respectively tangential) component is a polynomial of degree 1 (respectively 2);
%\item inside $E$ satisfy the biharmonic equation $\Delta^2v_{h}=0$;
\item $\bbP_2(E)\subseteq\bWE$.
\end{itemize}
We choose in $\bWE$ the degrees of freedom introduced in
\cite[Section 2.2]{ABSVsinum16}, namely:
\begin{itemize}
\item $\mathcal{W}_E^h$: The values of $v_h$ and $\nabla v_h$ at the vertices of $E$.
\end{itemize}
We note that as a consequence of the definition $\bWE$,
the degrees of freedom $\mathcal{W}_E^h$
are sufficient to uniquely determine $v_h$ and $\nabla v_h$
on the boundary of $E$.

We now present the global virtual space for the deflection:
for every decomposition $\CT_h$ of $\O$ into
simple polygons $E$, we define (without boundary conditions).
\begin{align*}
\label{uno_0}
W_h:=\{v_{h}\in \HdO:v_{h}|_{E}\in \bWE\quad \forall E\in \CT_h  
\}.
\end{align*}
In agreement with the local choice of the degrees of freedom, in $W_h$
we choose the following degrees of freedom:
\begin{itemize}
  \item $\mathcal{W}^h$: the values of $v_h$ and $\nabla v_h$ at the 
vertices of $\CT_h$.
 \end{itemize}
 
As a consequence of the definition of local virtual spaces $\bVE$ and $\bWE$,
we have the following result which will be used in the forthcoming analysis.
\begin{proposition}
\label{contenido}
Let $E$ be a simple polygon with $N_{E}$ edges.
Then $\nabla W_h^{E}\subseteq\bVE$.
\end{proposition}

\begin{proof}
Let $v_h\in\bWE$, then we have that:
$v_h\in \Hu2E$, $\Delta^{2}v_{h}=0$, $v_h|_{e}\in \bbP_3(e)$
and $\nabla v_h\cdot \boldsymbol{n}|_{e}\in \bbP_1(e)$ for all $e\in \partial E$.
Hence, $\nabla v_h\in [\HuE]^2$, $\nabla v_h\cdot \boldsymbol{t}|_{e}\in\bbP_2(e)$
and $\nabla v_h\cdot \boldsymbol{n}|_{e}\in \bbP_1(e)$ for all $e\in \partial E$, i.e,
$\nabla v_h|_{\partial E}\in \boldsymbol{\mathbb{B}}_{\partial E}$.
Moreover, $\rot(\nabla v_{h})=0\in \bbP_0(E)$.
On the other hand, we have that
$$0=\Delta^{2}v_{h}=\Delta\left(\Delta v_{h}\right)
=\Delta\left( \div(\nabla v_{h})\right)=\div\left(\Delta (\nabla v_{h})\right).$$
Since a star-shaped polygon $E$ is simply connected,
there exists $q\in\LE$ such that $\Delta (\nabla v_{h})=\Rot q$
(see \cite[Proposition VII.3.4]{BF}).
Thus, $\nabla v_{h}\in\bVE$. The proof is complete.
\end{proof}

Finally, once we have defined $\bfV_{h}$ and $W_{h}$,
we are able to introduce
our virtual element space $\bfX_h$. 
\begin{equation*}
\label{spaceh}
\bfX_h:=\{(v_h,\bta_h)\in W_h\times \bfV_h\}\cap 
\bfX.
\end{equation*}

\subsection{Bilinear forms and the loading term}

In this section we will discuss the construction of the discrete version
of the local bilinear forms $a^{E}(\cdot,\cdot)$ (cf \eqref{formasa})
and $b^{E}(\cdot,\cdot)$ (cf \eqref{formasb}), which will
be used to built the local bilinear form appearing in \eqref{localbili}.
Moreover, we will discuss the construction of the loading term
appearing in \eqref{forvar3}.

We define the projector
$\PiE:\bVE\longrightarrow [\mathbb{P}_1(E)]^2\subset \bVE$
for each $\bta_h\in\bVE$ as the solution of
\begin{align}
\label{proje_0}
\left\{\begin{array}{ll}
& a^E(\bp,\PiE\bta_h) = a^E(\bp,\bta_h)
% \disp\int_{E}\Cten\varepsilon(\bp)\colon\varepsilon(\PiE\bta_h)&=\disp\int_{E}
% \Cten\varepsilon(\bp)\colon\varepsilon(\bta_h)
\quad \forall \bp\in [\bbP_1(E)]^2,\\\\
& \left<\left<\bp,\PiE\bta_h\right>\right>=\left<\left<\bp,
\bta_h\right>\right>\quad\forall \bp\in ker(a^E(\cdot,\cdot)), 
\end{array}\right.
\end{align}
where for all ${\bf r}_h,{\bf s}_h$ in $\bVE$
$$
\left<\left<{\bf r}_h,{\bf s}_h\right>\right>:=\disp\dfrac{1}{N_{E}}\sum_{i=1}^{N_{E}}
{\bf r}_h(\vv_i)\cdot {\bf s}_h (\vv_i),\quad\vv_i=\text{ vertices of }E , \; 1\leq i\leq N_{E} .
$$
We note that the second equation in \eqref{proje_0}
is needed for the problem to be well-posed. In fact, it is easy to
check that it returns one (and only one) function
$\PiE\bta_h\in [\bbP_1(E)]^2$. Moreover,
we observe that the local degrees of freedom allow us to compute
exactly the right hand side of \eqref{proje_0}. Indeed,
for all $\bp\in [\bbP_1(E)]^2$, we have
\begin{align*}
\nonumber
a^{E}(\bp,\bta_{h})=\int_{E}\Cten\varepsilon(\bp)\colon\varepsilon(\bta_h)&=
-\int_{E}\Div(\Cten\varepsilon(\bp))\cdot \bta_h+\int_{\partial 
E}\left(\Cten\varepsilon(\bp) 
\bn \right)\cdotp\bta_h\\
&= \int_{\partial 
E}\left(\Cten\varepsilon(\bp) 
\bn \right)\cdotp\bta_h,
\label{proye}
\end{align*}
where we have used that $\Div(\Cten\varepsilon(\bp))=\0$.
Therefore, since the functions $\bta_h\in \bVE$ are known explicitly on the boundary,
the right hand side of \eqref{proje_0} can be computed exactly without knowing $\bta_h$
in the interior of $E$. 
As a consequence, the projection operator $\PiE$ is computable solely on the basis of the degrees of freedom values.

Let $\PiO:\bVE\to[\bbP_0(E)]^2$ be the $[L^2(E)]^2$-projector, defined  by
\begin{equation*}
\label{pi0}
\int_E\PiO\bta_h\cdot\bp_0=\int_E\bta_h\cdot\bp_0\quad\forall \bp_0\in[\bbP_0(E)]^2.
\end{equation*}
We note that as before, the right hand side above is computable.
In fact, we consider a simple polygon $E$ with barycenter $\bx_{E}=(x_{E},y_{E})^{t}$
and we have that any $\bp_0\in[\bbP_0(E)]^2$ can be written as
$\bp_0=\alpha(1,0)^{t}+\beta(0,1)^{t}=\alpha\Rot(y-y_{E})+\beta\Rot(x_{E}-x)$.
Thus, for all $\bta_h\in \bVE$ we have
\begin{align*}
\int_{E}\bta_h\cdot(1,0)^{t}&=\int_{E}\bta_h\cdot\Rot(y-y_{E})
= \int_{E}\rot \bta_h(y-y_{E})
-\int_{\partial 
E}\left(\bta_{h} 
\cdot\boldsymbol{t} \right)(y-y_{E})\\
&=\rot \bta_h\int_{E}(y-y_{E})
-\int_{\partial 
E}\left(\bta_{h} 
\cdot\boldsymbol{t} \right)(y-y_{E})
=-\int_{\partial 
E}\left(\bta_{h} 
\cdot\boldsymbol{t} \right)(y-y_{E}),
\end{align*}
where we have used that for $\bta_h\in \bVE$, $\rot\bta_h\in\bbP_0(E)$.
Using the same arguments, we get
\begin{align}
\nonumber
\int_{E}\bta_h\cdot(0,1)^{t}=-\int_{\partial 
E}\left(\bta_{h} 
\cdot\boldsymbol{t} \right)(x_{E}-x),
\end{align}
which shows that $\PiO\bta_h$ is computable solely on the basis of the degree of freedom values.
%\begin{remark}
%\label{prope}
%The virtual space $\bVE$ satisfies the following properties:
%\begin{itemize}
%\item $[\bbP_1(E)]^2\subseteq\bVE$;
%\item for all $\bta_h\in \bVE$, the local bilinear forms $a^{E}(\bp,\bta_{h})$ (cf.~\eqref{proye})
%and $b^{E}(\bta_{h},\bp_{0})$ (cf.~\eqref{pi0}), are directly computable in terms of the degrees of freedom
%of $\bta_h$, for all $\bp\in [\bbP_1(E)]^2$ and for all $\bp_{0}\in [\bbP_0(E)]^2$;
%\item for all $\bta_h\in \bVE$, the functions
%$\PiE\bta_h$ and $\PiO\bta_h$ can be explicitly
%computed from the degrees of freedom of $\bta_h$.
%\end{itemize}
%\end{remark}

Let now $S^E(\cdot,\cdot)$ and $S_0^E(\cdot,\cdot)$
be any symmetric
positive definite bilinear forms to be chosen as to satisfy 
\begin{align}
&\disp c_0a^{E}(\bta_h,\bta_h )\leq S^E(\bta_h,\bta_h)\leq\disp c_1a^{E}(\bta_h,\bta_h )
\quad\forall\bta_h\in\bVE 
\ \textrm{ with } \Pi_{\varepsilon}^E \bta_h = 0,\label{stabilS}\\
&\disp \tilde{c}_0 b^E(\bta_h,\bta_h)\leq S_0^E(\bta_h,\bta_h)\leq\disp \tilde{c}_1b^{E}(\bta_h,\bta_h)
\quad\forall\bta_h\in\bVE,
\label{stabilS0}
\end{align}
for some positive constants $c_0$, $c_1$, $\tilde{c}_0$ and $\tilde{c}_1$
depending only on the constant $C_{\CT}$ from mesh assumptions $\bA_1$ and $\bA_2$.
Then, we introduce on each element $E$ the local (and computable) bilinear forms 
\begin{equation*}
\begin{split}
a_h^E(\ga_h,\bta_h):=a^E(\PiE\ga_h,\PiE\bta_h)
+S^E(\ga_h-\PiE\ga_h,\bta_h-\Pi_{\varepsilon}^E\bta_h)\qquad \ga_h,\bta_h\in \bVE,\\
b_h^E(\ga_h,\bta_h):=b^E(\PiO\ga_h,\PiO\bta_h)+S_0^E(\ga_h-\Pi_{0}^E\ga_h,\bta_h-\PiO\bta_h)
\qquad \ga_h,\bta_h\in \bVE.
\end{split}
\end{equation*}

Now, we define in a natural way
\begin{equation*}
\label{bilinearforms}
a_{h}(\ga_h,\bta_h):=\sum_{E\in\CT_{h}}a_{h}^{E}(\ga_h,\bta_h),
\qquad b_{h}(\ga_h,\bta_h):=\sum_{E\in\CT_{h}}b_{h}^{E}(\ga_h,\bta_h)\qquad \ga_h,\bta_h\in \bfV_{h}.
\end{equation*}

The construction of $a_h^E(\cdot,\cdot)$ and $b_h^E(\cdot,\cdot)$
guarantees the usual consistency and stability properties of VEM,
as noted in the Proposition below.
Since the proof follows standard arguments
in the Virtual Element literature (see \cite{BBCMMR2013,BLRXX}) it is omitted.
\begin{proposition}
The local bilinear forms $a_h^E(\cdot,\cdot)$ and $b_h^E(\cdot,\cdot)$ on each element $E$ satisfy
\begin{itemize}
 \item Consistency: for all $h>0$ and for all $E\in\CT_h$ we have that
\begin{align}
\label{consis0}
a_h^E\left(\bp,\bta_h\right)&=a^E(\bp,\bta_h)\quad\forall\bp\in[\bbP_1(E)]^2,\; \forall\bta_h\in\bVE;\\
 \label{consis1}
b_h^E\left(\bp_0,\bta_h\right)&=b^E(\bp_0,\bta_h)\quad\forall\bp_0\in[\bbP_0(E)]^2,\; \forall\bta_h\in\bVE.
 \end{align}
 \item Stability: there exist positive constants
 $\alpha_{*}$, $\alpha^{*}$, $\beta_{*}$ and $\beta^{*}$, independent of $h$ and $E$, such that
\begin{align}
\alpha_{*}a^E(\bta_h,\bta_h)&\leq a_h^E(\bta_h,\bta_h)\leq\alpha^{*}a^E(\bta_h,\bta_h)\quad\forall 
\bta_h\in\bVE,\quad \forall E\in \CT_{h},\label{stab0}\\
\beta_{*}b^E(\bta_h,\bta_h)&\leq b_h^E(\bta_h,\bta_h)\leq\beta^{*}b^E(\bta_h,\bta_h)
\quad\forall\bta_h\in \bVE,\quad\forall E\in \CT_{h}.\label{stab1}
\end{align}
\end{itemize}
\end{proposition}
%\begin{proof}
%On the one hand, we have that \eqref{consis0} follows from the definition
%\eqref{bilineal_ae} and property \eqref{igual}. In the same way,
%\eqref{consis1} follows from the definition \eqref{bilineal_be} and property \eqref{pi0}.
%On the other hand, we now prove property \eqref{stab0}. In fact,
%as a consequence of \eqref{stabilS} and \eqref{proje_0} we have,
%$$a_h^E(\bta_h,\bta_h)\le a^E(\PiE\bta_h,\PiE\bta_h)+c_1 a^E(\bta_h-\PiE\bta_h,\bta_h-\PiE\bta_h)$$
%$$a_h^E(\bta_h,\bta_h)\le \max\{1,c_1\}(a^E(\PiE\bta_h,\PiE\bta_h)
%+a^E(\bta_h-\PiE\bta_h,\bta_h-\PiE\bta_h))=\alpha^{*}a^E(\bta_h,\bta_h).$$
%Similarly, for all $\bta_h\in \bVE$ and using \eqref{stabilS0} and \eqref{proje_0} we get,
%$$a_h^E(\bta_h,\bta_h)\ge \min\{1,c_0\}(a^E(\PiE\bta_h,\PiE\bta_h)
%+a^E(\bta_h-\PiE\bta_h,\bta_h-\PiE\bta_h))=\alpha_{*}a^E(\bta_h,\bta_h).$$
%Finally, \eqref{stab1} follows using the same arguments. The proof is complete.
%\end{proof}

We note that as a consequence of \eqref{stab0}
and \eqref{stab1}, the bilinear forms $a_h^E(\cdot,\cdot)$
and $b_h^E(\cdot,\cdot)$ are bounded with respect to the $H^1$ and $L^2$ norms, respectively.

We now discuss the construction of the loading term.
For every $E \in\CT_h$ we approximate the data
$g$ by a piecewise constant function $g_h$ on each element $E$
defined as the $L^{2}(E)$-projection of the load $g$
(denoted by $\bar{g}_{E}$). Let the loading term
\begin{equation}
\label{gh}
\left<g_h,v_h\right>_h:=\sum_{E\in\CT_h}\bar{g}_{E}\sum_{i=1}^{N_E}v_h(\vv_i)
\omega_{E}^{i}.
\end{equation}
where $\vv_1,\ldots,\vv_{N_E}$ are the vertices of $E$ and
$\omega_{E}^{1},\ldots,\omega_{E}^{N_E}$ are positive weights
chosen to provide the exact integral on $E$ when applied to linear
functions.

\subsection{Discrete problem}\label{method}

The results of the previous sections allow us to introduce
the discrete VEM in shear strain-deflection formulation
for the approximation of the
continuous Reissner-Mindlin formulation presented
in Problem~\ref{P_2}.

With this aim, we first note that since $\nabla\bWE\subset\bVE$ (see Proposition~\ref{contenido}),
the operator $\PiE$ can be also applied to $\nabla v_{h}$ for all $v_{h}\in \bWE$.
Hence, we introduce the following VEM discretization for the approximation of Problem~\ref{P_2}.
\begin{problem}
\label{P_4}
Find $(w_h,\ga_h)\in \bfX_h$ such that 
\begin{equation}
\label{forvar_3X}
a_h(\nabla w_h+\ga_h,\nabla v_h+\bta_h)+b_h(\ga_h,\bta_h)=\left<g_h,v_h\right>_h\qquad \forall 
(v_h,\bta_h)\in \bfX_h.
\end{equation}
\end{problem}

The next lemma shows that the problem above is coercive in the $\left||| \cdot ||\right|$ norm.
\begin{lemma}
\label{ha-elipt-disc}
There exists $\beta>0$, independent of $h$ and $t$ such that
$$
a_h(\nabla v_h+\bta_h,\nabla v_h+\bta_h)+b_h(\bta_h,\bta_h)\ge\beta\left|||(v_h,\bta_h)||\right|^2
\qquad\forall(v_h,\bta_h)\in \bfX_h.
$$
\end{lemma}
\begin{proof}
Thanks to \eqref{stab0}, \eqref{stab1} and \eqref{lax}, we have that
\begin{equation*}
a_h(\nabla v_h+\bta_h,\nabla v_h+\bta_h)+b_{h}(\bta_h,\bta_h)\geq C_{*}\left(a(\nabla v_h+\bta_h,\nabla 
v_h+\bta_h)+b(\bta_h,\bta_h)\right)\geq \beta\left|||(v_h,\bta_h)||\right|^2,
\end{equation*}
with $\beta:=\min\left\{C_{*},\alpha\right\}$.
\end{proof}

We deduce immediately from Lemma \ref{ha-elipt-disc} that Problem \ref{P_4} is well-posed.

\begin{remark}
The solution of Problem~\ref{P_2} delivers the shear strain and deflection. In addition,
it is possible to readily obtain the rotations $\bthe$ by recalling \eqref{change}.
At the discrete level, this strategy corresponds to computing the rotations
as a post-processing of the shear strain and deflection.
If $(w_h,\ga_h)\in \bfX_h$ is the unique solutions of Problem~\ref{P_4}, then the function
$$\bthe_h=\nabla w_h+\ga_h,$$
is an approximation of the rotations.
The accuracy of such approximation will be established in the following section.
\end{remark}

\section{Convergence analysis}
\label{SEC:approximation}

In the present section, we develop an error analysis
for the discrete virtual element scheme presented
in Section~\ref{method}.
For the forthcoming analysis, we will assume
that the mesh assumptions $\bA_1$ and $\bA_2$,
introduced in Section~\ref{meshassup}, are satisfied.

For the analysis we will introduce the broken $H^{1}$-norm:
$$\|v\|_{1,h,\O}^{2}:=\sum_{E\in\CT_h}\|v\|_{1,E}^{2},$$
which is well defined for every $v\in L^{2}(\O)$ such that
$v|_{E}\in \HuE$ for all polygon $E\in \CT_{h}$.

Moreover, we recall the following result
which are derived by interpolation
between Sobolev spaces (see, for instance \cite{Bergh-Lofstrom})
from the analogous result for integer values of $s$.
In its turn, the result for integer values is stated
in \cite[Proposition 4.2]{BBCMMR2013} and follows from the
classical Scott-Dupont theory (see \cite{BS-2008}).

\begin{proposition}
\label{estima1}
There exists a constant $C>0$,
such that for every $v\in[\HsE]^d$, $d=1,2$ there exists
$v_{\Pi}\in [\bbP_k(E)]^d$, $k\geq 0$ such that
 \begin{eqnarray*}
\arrowvert v-v_{\Pi}\arrowvert_{l,E}\leq C h_E^{s-l}|v|_{s,E}\quad 0\leq s \leq k+1, l=0,\ldots,[s].
\end{eqnarray*}
with $[s]$ denoting largest integer equal or smaller than $s \in {\mathbb R}$.
\end{proposition}
%\begin{proposition}
%There exists a constant $C>0$,  such that for every
%$\bta\in [\HsE]^2$ there exists $\bta_{\Pi}\in [\bbP_k(E)]^2$,
%$k\geq 0$ such that
%\begin{eqnarray}\label{estima2}
%\arrowvert \bta-\bta_{\Pi}\arrowvert_{l,E}\leq C h_E^{s-l}|\bta|_{s,E}\quad 0\leq s\leq k+1, l=0,\ldots,s.
%\end{eqnarray}
%\end{proposition}

The first step is to establish the following result.

\begin{lemma}\label{cotaa}
Let $(w,\ga)\in\bfX$ be the unique solution
to the continuous Problem~\ref{P_2} 
and let $\bthe:=\nabla w+\ga$.
Let $(w_h,\ga_h)\in \bfX_h$ be the unique solution
to the discrete Problem~\ref{P_4}. Then,
for any $(w_{I},\ga_{I})\in \bfX_h$  
and $(\bthe_{\Pi},\ga_0)\in[\LO]^4$
such that $\bthe_{\Pi}|_{E}\in[\bbP_{1}(E)]^2$ and $\ga_{0}|_{E}\in[\bbP_{0}(E)]^2$
for all $E\in\CT_h$, there exists $C>0$ independent of $h$ and $t$ such that 
 \begin{align*}
 |||w-w_h,\ga-\ga_h|||&\leq C \left(t^{-1}\left(\|\ga-\ga_I\|_{0,\O}+\|\ga_{0}-\ga\|_{0,\O}\right)
 +\|\ga-\ga_I\|_{1,\O}+h\Vert g\Vert_{0,\O}\right.\\
&\hspace{4.5cm}\left.+\|\bthe-\bthe_{\Pi}\|_{1,h,\O}+\|\nabla w-\nabla w_I\|_{1,\O}\right).
 \end{align*}
\end{lemma}
\begin{proof}
We set $\boldsymbol{\delta}_{\ga}:=\ga_h-\ga_I$, $\delta_{w}:=w_h-w_I$,
$\bthe_h:=\nabla w_h+\ga_h$, $\bthe_I:=\nabla w_I+\ga_I$
and $\boldsymbol{\delta}_{\bthe}:=\bthe_h-\bthe_I$.
Thanks to Lemma~\ref{ha-elipt-disc} and equations \eqref{forvar_3X},
\eqref{consis0}, \eqref{consis1} we have that
%and \eqref{gh}
\begin{align*}
\beta|||(w_h-w_I),(\ga_h-\ga_I)|||^2&\leq 
a_h(\bthe_h-\bthe_I,\boldsymbol{\delta}_{\bthe})+b_h(\ga_h-\ga_I,\boldsymbol{\delta}_{\ga})\\
\nonumber
&= a_h(\nabla w_h+\ga_h,\boldsymbol{\delta}_{\bthe})+b_h(\ga_h,\boldsymbol{\delta}_{\ga})-\left( 
a_h(\bthe_I,\boldsymbol{\delta}_{\bthe})+b_h(\ga_I,\boldsymbol{\delta}_{\ga})\right)\\
\nonumber
&=\left<g_h,\delta_{w}\right>_{h}
-\sum_{E\in\CT_h}\left(a_h^E(\bthe_I-\bthe_{\Pi},\boldsymbol{\delta}_{\bthe})
+a^E(\bthe_{\Pi}-\bthe,\boldsymbol{\delta}_{\bthe})
+a^E(\bthe,\boldsymbol{\delta}_{\boldsymbol{\bthe}})\right)\\
&\hspace{0.5cm}\qquad-\sum_{E\in \CT_h}\left(b_h^E(\ga_I-\ga_{0},\boldsymbol{\delta}_{\ga})+b^E(\ga_{0}-\ga,\boldsymbol{\delta}_{\ga} )+b^E(\ga,\boldsymbol{\delta}_{ \ga } )\right)\\
&\leq T_1+T_2+T_3,
\end{align*}
where 
\begin{align*}
T_1&:=\left|\left<g_h,\delta_{w}\right>_h-\left(g,\delta_{w}\right)_{0,\O}\right|,\qquad
T_2:=\left|\sum_{E\in \CT_h}\left(a_h^E(\bthe_I-\bthe_{\Pi},\boldsymbol{\delta}_{\bthe})
-a^E(\bthe_{\Pi}-\bthe,\boldsymbol{\delta}_{\bthe})\right)\right|,\\
T_3:&=\left|\sum_{E\in \CT_h}\left(b_h^E(\ga_I-\ga_{0},\boldsymbol{\delta}_{\ga})
-b^E(\ga_{0}-\ga,\boldsymbol{\delta}_{\ga} )\right)\right|.
\end{align*}
We now bound each term $T_i$, $i=1,2,3$, with a constant $C$ independent of $h$ and $t$.

First, we bound the term $T_{2}$. Using \eqref{stab0}, the fact
that bilinear form $a(\cdot,\cdot)$ is bounded
and finally adding and subtracting $\bthe$, we obtain
\begin{align*}
T_2 &\leq\sum_{E\in\CT_h}\left|a_h^E(\bthe_I-\bthe_{\Pi},\boldsymbol{\delta}_{\bthe})\right|
+\sum_{E\in\CT_h}\left|a^E(\bthe_{\Pi}-\bthe,\boldsymbol{\delta}_{\bthe})\right|\\
&\leq\sum_{E\in\CT_h}C(\|\bthe_I-\bthe_{\Pi}\|_{1,E}
+\|\bthe_{\Pi}-\bthe\|_{1,E})\|\boldsymbol{\delta}_{\bthe}\|_{1,E}\\
&\leq\sum_{E\in\CT_h}C(\|\bthe_I-\bthe\|_{1,E}
+\|\bthe_{\Pi}-\bthe\|_{1,E})\|\boldsymbol{\delta}_{\bthe}\|_{1,E}.
\end{align*}

For the term $T_{3}$, using \eqref{stab1}, the definition
of bilinear form $b(\cdot,\cdot)$, the Cauchy--Schwarz inequality,
and finally adding and subtracting $\ga$, we obtain
\begin{align*}
T_3&\leq\sum_{E\in \CT_h}C(\|\ga_I-\ga\|_{0,E}+\|\ga_{0}-\ga\|_{0,E})t^{-2}\| \boldsymbol{\delta}_{\ga} \|_{0,E}.
\end{align*}

Now, we bound $T_1$. Using the definition \eqref{gh},
and adding and subtracting $\bar{g}_{E}$ we rewrite the term as follows
\begin{align*}
T_1&=\left| \sum_{E\in\CT_h}\left(\bar{g}_{E}\sum_{i=1}^{N_E}\delta_{w}(\vv_i)\omega_E^{i}\right)
-\sum_{E\in\CT_h}\int_{E}g\delta_{w}\right|\\
&= \left| \sum_{E\in\CT_h}\left(\bar{g}_{E}\sum_{i=1}^{N_E}\delta_{w}(\vv_i)\omega_E^{i}
-\int_{E}\bar{g}_{E}\delta_{w}\right)
+\sum_{E\in\CT_h}\left(\int_{E}(\bar{g}_{E}-g)(\delta_{w}-p)\right)\right|,
\end{align*}
for any $p\in \bbP_{0}(E)$, where we have used the definition of $\bar{g}_{E}$.
Therefore,
\begin{align*}
T_1&\leq \left| \sum_{E\in\CT_h}\left(\bar{g}_{E}\sum_{i=1}^{N_E}\delta_{w}(\vv_i)\omega_E^{i}
-\int_{E}\bar{g}_{E}\delta_{w}\right)\right|
+\sum_{E\in\CT_h}\|g-\bar{g}_{E}\|_{0,E}\|\delta_{w}-p\|_{0,E}:=T_{1}^{a}+T_{1}^{b}.
\end{align*}
First, $T_{1}^{b}$ is easily bounded. In fact, taking $p$ as in
Proposition~\ref{estima1}, we obtain that
\begin{equation*}\label{part1}
T_{1}^{b}\le Ch\Vert g\Vert_{0,\O}\|\delta_{w}\|_{1,\O}.
\end{equation*}
In what follows
we will manipulate the terms $T_{1}^{a}$:
adding and subtracting $p_{0}\in \bbP_{0}(E)$,
and since the integration rule in \eqref{gh} is exact for constant functions, we have
\begin{align}
T_{1}^{a}\le&\left|\sum_{E\in\CT_h}\int_{E}\bar{g}_{E}(\delta_{w}-p_{0})\right|
+\left|\sum_{E\in\CT_h}\left(\bar{g}_{E}\left(\sum_{i=1}^{N_E}
(\delta_{w}-p_{0})(\vv_i)\omega_E^{i}\right)\right)\right|\nonumber\\
&\leq \|g\|_{0,\O}\left(\sum_{E\in\CT_h}\|\delta_{w}-p_{0}\|_{0,E}^{2}\right)^{1/2}
+\sum_{E\in\CT_h}|E|\bar{g}_{E}\|\delta_{w}-p_{0}\|_{L^\infty(\partial E)}\nonumber\\
&\leq \|g\|_{0,\O}\left(\sum_{E\in\CT_h}\|\delta_{w}-p_{0}\|_{0,E}^{2}\right)^{1/2}
+\|g\|_{0,\O}\left(\sum_{E\in\CT_h}h_{E}^{2}\|\delta_{w}-p_{0}\|_{L^\infty(\partial E)}^{2}\right)^{1/2}.\label{eqfinal}
\end{align}
% ---------------------------------------------------------------------------------
% I CHANGED THIS PROOF, THERE WAS A FISHY POINT
%
%Now, for any $E\in\CT_h$, let $\bx_{E}$ be the center of the ball
%with respect to which $E$ is starred. Let $T$ be the
%smallest triangle containing the ball
%of center $\bx_{E}$ and radius $h_E$. Then, by Calderon extension
%theorem, there exists $\bar{\delta}_w\in H^1(T)$
%such that $\bar{\delta}_{w}|_{E}=\delta_{w}$ and
%$\|\bar{\delta}_w\|_{1,T}\leq C \|\delta_{w}\|_{1,E}$.
%Moreover, there exists $\bar{p}_0\in \bbP_{0}(T)$,
%such that  $\bar{p}_0|_{E}=p_0$. Hence, by a scaling argument in the
%reference triangular element $\widehat{T}$, allow us to conclude that
%%
%\begin{align*}
%\|\delta_{w}-p_{0}\|_{\infty,E}\leq&\|\bar{\delta}_{w}-\bar{p}_{0}\|_{\infty,T}
%= \|\hat{\delta}_{w}-\hat{p}_{0}\|_{\infty,\widehat{T}}\\
%\leq&C\|\hat{\delta}_{w}-\hat{p}_{0}\|_{0,\widehat{T}}\leq
%Ch_{T}^{-1}\|\bar{\delta}_{w}-\bar{p}_{0}\|_{0,T}\leq
%C\|\bar{\delta}_{w}\|_{1,T}\leq C\|\delta_w\|_{1,E},
%\end{align*}
%where we have used that $\|\bar{\delta}_{w}-\bar{p}_{0}\|_{0,T}\leq
%Ch_{T}\|\bar{\delta}_{w}\|_{1,T}$ with $\bar{p}_{0}$ as in
%Proposition~\ref{estima1}. Moreover, we have that
%$$\|\delta_{w}-p_{0}\|_{0,E}\le \|\bar{\delta}_{w}-\bar{p}_{0}\|_{0,T}\leq
%Ch_{T}\|\bar{\delta}_{w}\|_{1,T}\le Ch_{E}\|\delta_w\|_{1,E}$$
% -----------------------------------------------------------------------------
Now, we fix $p_{0}:=\Pi_{\partial E}^0(\delta_{w})=\frac{1}{\vert\partial E\vert}\int_{\partial E}\delta_{w}$.
Thus, we have that $\delta_{w}-p_{0}$ is a (continuous) piecewise polynomial on $\partial E$,
and that the length of the edges of $E$ is bounded from below in the sense of assumption ${\bf A}_1$.
Therefore, we can apply Lemma 3.1 in \cite{bertoluzza}, standard polynomial approximation estimates
and a trace inequality to derive the following estimate for the second term on the right hand side in \eqref{eqfinal}:
$$
\|\delta_{w}-p_{0}\|_{L^\infty(\partial E)} \le C | \delta_{w} |_{1/2,\partial E} + h_E^{-1/2} \| \delta_{w}-p_{0} \|_{0,\partial E} 
\le C  | \delta_{w} |_{1/2,\partial E} \le C  | \delta_{w} |_{1,E}.
$$
For the first term on the right hand side
in \eqref{eqfinal}, we consider $c\in\bbP_{0}(E)$
such that Proposition~\ref{estima1} holds with respect
to $\delta_{w}$ (for instance, take $c$ as the average
of $\delta_{w}$ on $E$). Thus, simple calculations yield
\begin{equation*}
\begin{split}
\|\delta_{w}-p_{0}\|_{0,E}&\le \|\delta_{w}-c\|_{0,E}+\Vert\Pi_{\partial E}^0(\delta_{w}-c)\Vert_{0,E}\\
&\le Ch_E\vert\delta_{w}\vert_{1,E}+h_E^{1/2}\Vert\Pi_{\partial E}^0(\delta_{w}-c)\Vert_{0,\partial E}\\
&\le Ch_E\vert\delta_{w}\vert_{1,E}+h_E^{1/2}\Vert\delta_{w}-c\Vert_{0,\partial E}\\
&\le Ch_E\vert\delta_{w}\vert_{1,E}+\Vert\delta_{w}-c\Vert_{0,E}+h_E\vert\delta_{w}\vert_{1,E},\\
&\le Ch_E\vert\delta_{w}\vert_{1,E},
\end{split}
\end{equation*}
where we have used a scaled trace estimate on polygons
(also sometimes called Agmon inequality in the FEM literature),
see for instance \cite[Lemma~14]{BMRR}).
%%%%%%
Hence, from the above estimates, we obtain,
\begin{align*}
T_{1}^{a}\le& Ch\|g\|_{0,\O} | \delta_w |_{1,\O}.
\end{align*}
Thus, since $|\delta_{w}|_{1,\O}\leq |||\delta_{w},\boldsymbol{\delta}_{\ga} |||$, we have that
\begin{equation}
\label{cotag}
T_1\leq T_1^{a}+T_1^{b}\leq Ch\Vert g\Vert_{0,\O}|||\delta_{w},\boldsymbol{\delta}_{\ga} |||.
\end{equation}

Therefore, by combining \eqref{cotag} with the above bounds for $T_2$ and $T_3$, we get
\begin{align*}
|||(w_h-w_I),(\ga_h-\ga_I)|||\leq &C\big(t^{-1}(\|\ga-\ga_I\|_{0,\O}
+\|\ga_{0}-\ga\|_{0,\O})\\
&+\|\bthe-\bthe_I\|_{1,\O}+\|\bthe_{\Pi}-\bthe\|_{1,h,\O}+h\Vert g\Vert_{0,\O}\big).
\end{align*}
Hence, the proof follows from the bound above,
the triangular inequality,
the definition of $|||\cdot|||$ (see \eqref{norm}),
the definition of $\bthe_I$ and the inequality
$\|\bthe-\bthe_{I}\|_{1,\O}\leq \|\nabla w-\nabla w_I\|_{1,\O}+\|\ga-\ga_I\|_{1,\O}$.
In fact,
\begin{align*}
|||w-w_h,\ga-\ga_h|||\leq& |||w-w_I,\ga-\ga_I|||+|||w_I-w_h,\ga_I-\ga_h|||\\
\leq&C ( t^{-1}\|\ga-\ga_I\|_{0,\O}+t^{-1}\|\ga_{0}-\ga\|_{0,\O}
+\|\ga-\ga_I\|_{1,\O}\\
&+h\Vert g\Vert_{0,\O}+\|\bthe-\bthe_{\Pi}\|_{1,h,\O}
+\|\nabla w-\nabla w_I\|_{1,\O}).
\end{align*}
The proof is complete.
\end{proof}

The next step is to find appropriate terms $(w_{I},\ga_{I})$,
$(w_{\Pi},\ga_{\Pi})$ and $\ga_{0}$ that can be used in
Lemma~\ref{cotaa} to prove the claimed convergence.
% -----------------------------------------
As a preliminary construction, we introduce, for every vertex $\vv$
of the mesh laying on $\partial\O$, the following function. 
Let $e_\vv$ be any one of the two edges on $\partial\O$ sharing $\vv$,
fixed once and for all; the only rule being that, if one of the two edges
is in $\G_c$ and the other is not, then the one in $\G_c$ must be chosen.
Then, we denote by ${\boldsymbol\varphi}_{\vv}$ the unique (vector valued)
polynomial of degree $2$ living on $e_\vv$ such that
\begin{equation}\label{phifun}
\int_{e_\vv} {\bf p} \cdot {\boldsymbol\varphi}_{\vv} = {\bf p}(\vv) \qquad \forall {\bf p} \in [\bbP_2(e_\vv)]^{2}.
\end{equation}
% ----------------------------------------

Then, for the term $w_{I}\in W_{h}$, we have the following result.
\begin{proposition}
\label{estima3}
There exists a positive constant $C$, such that
for every $v\in H^{3}(\O)$ there exists $v_I\in W_{h}$
that satisfies
\begin{eqnarray*}
\vert v-v_{I}\vert_{l,\O}\leq C {h}^{3-l}|v|_{3,\O},\quad l=0,1,2.
\end{eqnarray*}
\end{proposition}
\begin{proof}
 Given $v\in H^{3}(\O)$,
we consider $v_{\Pi}\in\LO$ defined on each $E\in\CT_h$ so that
 $v_{\Pi}|_{E}\in\bbP_2(E)$ and the estimate of Proposition~\ref{estima1}
 holds true.
 
 For each polygon $E\in\CT_h$, consider the triangulation $\CT_h^{E}$
 obtained by joining each vertex of $E$ with the center of the ball
in assumption $\bA_2$. Let
 $\hCT_h:=\bigcup_{E\in\CT_h}\CT_h^{E}$. Since we are assuming
 $\bA_1$ and $\bA_2$, $\big\{\hCT_h\big\}_h$ is a shape-regular family of
 triangulations of $\O$.
 
Let $v_{\SZ}$ be the reduced Hsieh-Clough-Tocher triangle
(see \cite{ciarlet,ciarlet2}) interpolant of $v$ over $\hCT_h$, slightly modified as follows. 
% ----
For the nodes on the boundary, the value of $\nabla v_{\SZ}$ is given by 
$$
\nabla v_{\SZ}(\vv) := \int_{e_\vv} {\nabla v} \cdot {\boldsymbol\varphi}_{\vv} ,
$$
see \eqref{phifun}, while the values of the remaining degrees of freedom is the same as in the original version.
This is a modification, in the spirit of the Scott-Zhang interpolation \cite{scott-zhang}, of the standard nodal value; 
the motivation for such modification is not related directly to the present result (that would hold also with the original HCT interpolant) and will be clearer in the sequel. 
This modified version still satisfies similar approximation properties with respect the original version \cite{ciarlet,ciarlet2}; we omit the standard proof and simply state the result: 
 \begin{equation}
 \label{err_SZ}
\vert v-v_{\SZ}\vert_{l,\O}
 \le C{h}^{3-l}\left|v\right|_{3,\O}\quad l=0,1,2.
 \end{equation}
Now, for each $E\in\CT_h$, we define $v_I|_{E}\in\Hu2E$ as the solution
of the following problem:
$$
\left\{\begin{array}{l}
-\Delta^2 v_I=0
\quad\textrm{in }E,
\\[0.1cm]
\hphantom{-\Delta^2}
v_I=v_{\SZ}
\quad\textrm{on }\partial E,
\\[0.1cm]
\hphantom{-\Delta^2}
\partial_{\boldsymbol{n}} v_I=\partial_{\boldsymbol{n}} v_{\SZ}
\quad\textrm{on }\partial E.
\end{array}\right.
$$
Note that $v_I|_{E}\in W_{h}^E$. Moreover, although $v_I$ is defined
locally, since on the boundary of each element it coincides with
$v_{\SZ}$ which belongs to $\HdO$, we have that also $v_I$ belongs to
$\HdO$ and, hence, $v_I\in W_h$.

According to the above definition we have that
$$
\left\{\begin{array}{l}
-\Delta^2(v_{\Pi}-v_I)=0
\quad\textrm{in } E,
\\[0.1cm]
\hphantom{-\Delta^2}
v_{\Pi}-v_I
=v_{\Pi}-v_{\SZ}
\quad\textrm{on }\partial E,
\\[0.1cm]
\hphantom{-\Delta}
\partial_{\boldsymbol{n}}(v_{\Pi}-v_I)
=\partial_{\boldsymbol{n}}(v_{\Pi}-v_{\SZ})
\quad\textrm{on }\partial E,
\end{array}\right.
$$
and, hence, it is easy to check that
\begin{align*}
\left|v_{\Pi}-v_I\right|_{2,E}&=\inf\left\{\left|z\right|_{2,E},\ z\in \Hu2E:
\ z=v_{\Pi}-v_{\SZ}\ \mbox{ on }\partial E \text{ and }\
\partial_{\boldsymbol{n}}z=\partial_{\boldsymbol{n}}(v_{\Pi}-v_{\SZ})\ \mbox{ on }\partial E\right\}
\\
&\le\left|v_{\Pi}-v_{\SZ}\right|_{2, E}.
\end{align*}
Therefore,
\begin{align*}
\left|v-v_I\right|_{2,E}
&\le\left|v-v_{\Pi}\right|_{2,E}
+\left|v_{\Pi}-v_I\right|_{2,E}\\
&\le\left|v-v_{\Pi}\right|_{2,E}
+\left|v_{\Pi}-v_{\SZ}\right|_{2,E}\\
&\le 2\left|v-v_{\Pi}\right|_{2,E}
+\left|v-v_{\SZ}\right|_{2,E}\\
&\le Ch_E\left|v\right|_{3,E}
+\left|v-v_{\SZ}\right|_{2,E},
\end{align*}
where we have used Proposition~\ref{estima1}.
By summing on all the elements and recalling \eqref{err_SZ}
(plus standard approximation estimates for
polynomials on polygons) we obtain
$$
\left|v-v_I\right|_{2,\O} \le C \big(h\left|v\right|_{3,\Omega}
+\left|v-v_{\SZ}\right|_{2,\Omega} \big) \le C h |v|_{3,\Omega}.
$$
Moreover, from the above bound and (recalling that
$\partial_{\boldsymbol{n}}(v_{I}-v_{\SZ})=0$ and $(v_{I}-v_{\SZ})=0$ on $\partial E$)
a Poincar\'e-type inequality, we have
\begin{align*}
\left|v-v_I\right|_{1,E}
& \le\left|v-v_{\SZ}\right|_{1,E}
+\left|v_{\SZ}-v_I\right|_{1,E}
\le\left|v-v_{\SZ}\right|_{1,E}
+Ch_{E}\left|v_{\SZ}-v_I\right|_{2,E}
\\
& \le\left|v-v_{\SZ}\right|_{1,E}
+Ch_{E}\left|v-v_{\SZ}\right|_{2,E}
+Ch_{E}\left|v-v_I\right|_{2,E},
\end{align*}
so that, summing on all the elements and using the bounds above,
$$
\left|v-v_I\right|_{1,\Omega} \le C h^{2}\left|v\right|_{3,\Omega}.
$$
By an analogous argument one obtains
\begin{align*}
\|v-v_I\|_{0,\Omega}\le C \big( \|v-v_{\SZ}\|_{0,\Omega}
+ h \left|v_{\SZ}-v_I\right|_{1,\Omega} \big)
 \le C h^{3}\left|v\right|_{3,\Omega},
\end{align*}
which allows us to complete the proof.
\end{proof}

Finally, we present the following result for
the approximation properties of the space $\bfV_h$.
\begin{proposition}
\label{estima4}
There exists $C>0$ such that for every $\bta\in [\HsO]^2$ with $s\in[1,2]$ there exists  
$\bta_I\in \bfV_h$ that satisfies
\begin{eqnarray*}
\Vert \bta-\bta_{I}\Vert_{0,\O}+h\vert\bta-\bta_{I}\vert_{1,\O}
\leq C h^{s}|\bta|_{s,\O}.
\end{eqnarray*}
\end{proposition}
\begin{proof}
We refer the reader to Section~\ref{sec:added:X} for the definition of the degrees
of freedom of $\bfV_h$ and define $\bta_{I}$ as follows. All degrees of freedom
associated to internal vertices are calculated as an integral average of
$\bta$ on the elements sharing the vertex (as in standard Cl\'ement interpolation).
All the vertex boundary values are taken as (see \eqref{phifun})
$$
\bta_I(\vv) = \int_{e_\vv} {\bta} \cdot {\boldsymbol\varphi}_{\vv}.
$$
Finally, the edge degrees of freedom are computed directly by 
$$
\disp \dfrac{1}{|e|}\int_{e}\bta_I\cdot\boldsymbol{t}=
\disp \dfrac{1}{|e|}\int_{e}\bta\cdot\boldsymbol{t}
\quad \forall \text{ edge }e\in\CT_h.
$$
The rest of the proof is omitted since it follows repeating essentially the same argument used
to establish \cite[Proposition~4.1]{BLV}.
\end{proof}

According to the above results, we are able to establish
the convergence of the Virtual Element scheme presented in Problem~\ref{P_4}.
\begin{theorem}
\label{aproximation2}
Let $(w,\ga)\in\bfX$ and $(w_h,\ga_h)\in \bfX_h$ be the unique solutions
of the continuous and discrete problems, respectively.
Assume that $(w,\ga)\in(H^{3}(\O),[\HdO]^{2})$.
Then, there exists $C>0$
independent of $h$, $g$ and $t$ such that
\begin{equation*}
|||w-w_h,\ga-\ga_h|||\leq C h\left(t^{-1}|\ga|_{1,\O}+|\bthe|_{2,\O}+|w|_{3,\O}+\Vert g\Vert_{0,\O}\right),
\end{equation*}
where  $\bthe:=\nabla w+\ga$.
\end{theorem}
\begin{proof}
The proof follows from Lemma~\ref{cotaa} and
Propositions~\ref{estima1}, \ref{estima3}
and \ref{estima4}. In fact, 
 \begin{align*}
|||w-w_h,\ga-\ga_h|||\leq & C\Big(t^{-1}(\|\ga-\ga_I\|_{0,\O}+\|\ga_{0}-\ga\|_{0,h,\O})+\|\ga-\ga_I\|_{1,\O}\\
&+h\Vert g\Vert_{0,\O}+\|\bthe-\bthe_{\Pi}\|_{1,h,\O}+\|\nabla w-\nabla w_I\|_{1,\O}\Big)\\
\leq& C h\left(t^{-1}|\ga|_{1,\O}+|\bthe|_{2,\O}+|w|_{3,\O}+\Vert g\Vert_{0,\O}\right),
\end{align*}
where we have used that $\ga=\bthe-\nabla w$
so that $|\ga|_{2,\O}\leq |w|_{3,\O}+|\bthe|_{2,\O}$. Thus, we conclude the proof. 
\end{proof}

% -------------
\begin{remark}
It is easy to check that the couple $(w_{I},\ga_{I})$ used in Theorem~\ref{aproximation2}
(accordingly to the interpolants definition given in Propositions~\ref{estima3} and \ref{estima4})
does actually satisfy the boundary conditions and is thus in $\bfX_h$. 
Indeed, the condition $w_I=0$ on $\G_c\cup\G_s$ follows immediately
from the analogous one for $w$. The condition $\nabla w_I + \ga_I=\0$ on $\G_c$
can be easily derived from the analogous one for $(w,\ga)$ combined with our choice
for the boundary node interpolation and the definition of the discrete spaces.
\end{remark}
% -------------

\begin{remark}
We note that Theorem~\ref{aproximation2} provides also an error
estimate for the rotations in $H^1(\O)$-norm.
\end{remark}

In what follows, we restrict our analysis
considering clamped boundary conditions on the whole
boundary, essentially to exploit the associated regularity properties
of the continuous solution of the Reissner-Mindlin equations.
Nevertheless, the analysis in what follows can be straightforwardly
extended to other boundary conditions.

Now, we present the following result which
establish an improve error estimate for rotations in $\LO$-norm
and the deflection in $H^1(\O)$-norm.

%%%%%
%\medskip
%{\bf NOTE FOR YOU: in the proposition below we ask for  $\G_c=\G$
%(essentially to have the regularity result), could you check and just
%add a Remark saying something about more general boundary conditions?
%like this is a little limited}
%\medskip
%%%%%

\begin{proposition}
\label{aproximation3}
Assume that the hypotheses of Theorem~\ref{aproximation2} hold.
Moreover, assume that the domain $\O$ be either regular,
or piecewise regular and convex, that $g\in H^{1}(E)$ for all $E\in\CT_h$
and that $\G_c=\G$.
Then, for any $(w_{\Pi},\ga_{\Pi},\ga_0)\in[\LO]^{5}$
such that $w_{\Pi}|_{E}\in\bbP_{2}(E)$,
$\ga_{\Pi}|_{E}\in[\bbP_{1}(E)]^2$ and $\ga_{0}|_{E}\in[\bbP_{0}(E)]^2$
for all $E\in\CT_h$, there exists $C>0$
independent of $h$, $g$ and $t$ such that
\begin{align}
&\|\bthe-\bthe_h\|_{0,\O}\leq C(h+t)\left(|||w-w_{h},\ga-\ga_{h}|||+h\Vert g\Vert_{1,h,\O}
+\|\nabla w-\nabla w_{\Pi}\|_{1,h,\O}\right.\label{frgt1}\\
&\hspace{6.7cm}\left.+\|\ga-\ga_{\Pi}\|_{1,h,\O}+t^{-1}\|\ga-\ga_{0}\|_{0,\O}\right);\nonumber\\
&\|w-w_h\|_{1,\O}\leq C(\|\bthe-\bthe_h\|_{0,\O}+\|\ga-\ga_h\|_{0,\O}).\label{frgt2}
\end{align}
\end{proposition}
\begin{proof}
The core of the proof is based on a duality argument.
We first establish \eqref{frgt1}.
We begin by introducing the following well-posed auxiliary problem:
Find $(\widetilde{w},\widetilde{\ga})\in \bfX$ such that 
\begin{equation}
\label{P_aux}
a(\nabla \widetilde{w}+\widetilde{\ga},\nabla v+\bta)
+b(\widetilde{\ga},\bta)=(\bthe-\bthe_h,\nabla v+\bta)_{0,\O}\quad \forall (v,\bta)\in \bfX.
\end{equation}
The following regularity result for the solution of problem above
holds (see \cite[Theorem~2.1]{LNS22}):
\begin{equation}\label{regular}
 \|\widetilde{w}^1\|_{3,\O}+t^{-1}\|\widetilde{w}^2\|_{2,\O}
 +t^{-1}\|\widetilde{\ga}\|_{1,\O}\leq C\|\bthe-\bthe_h\|_{0,\O},
\end{equation}
where $\widetilde{w}^1$ is the solution
of the Kirchhoff limit problem and $\widetilde{w}^2:=\widetilde{w}-\widetilde{w}^1$.
Let $(\widetilde{w}^{1}_{I},\widetilde{\ga}_{I})\in\bfX_h$ be
the interpolant of $(\widetilde{w}^{1},\widetilde{\ga})$
given by Propositions~\ref{estima3} and \ref{estima4},
respectively.
Therefore, the above regularity result yield immediately:
\begin{align}
\label{cota1}
\|\widetilde{w}^1-\widetilde{w}^{1}_{I}\|_{1,\O}
+h\|\widetilde{w}^1-\widetilde{w}^{1}_{I}\|_{2,\O}
+t^{-1}h\|\widetilde{\ga}-\widetilde{\ga}_{I}\|_{0,\O}\leq h^{2}\|\bthe-\bthe_h\|_{0,\O},\\
\label{cota2}
\|\widetilde{w}^2\|_{2,\O}+\|\widetilde{\ga}
-\widetilde{\ga}_{I}\|_{1,\O}\leq t\|\bthe-\bthe_h\|_{0,\O}.
\end{align}
Next, choosing $v:=(w-w_{h})$ and $\bta=(\ga-\ga_{h})$ in \eqref{P_aux},
so that $\nabla v + \bta = \bthe - \bthe_h$,
and then adding and subtracting the term
$\nabla \widetilde{w}^{1}_I+\widetilde{\ga}_{I}$,
we obtain
\begin{align}
\nonumber
\|\bthe-\bthe_h\|_{0,\O}^2 &=a(\bthe-\bthe_h,\nabla \widetilde{w}+\widetilde{\ga}
-\nabla \widetilde{w}_{I}^1-\widetilde{\ga}_{I})+a(\bthe-\bthe_h,\nabla\widetilde{w}_{I}^{1}
+\widetilde{\ga}_{I})\\
\nonumber
&\hspace{0.4cm}+b(\ga-\ga_{h},\widetilde{\ga}-\widetilde{\ga}_{I})
+b(\ga-\ga_{h},\widetilde{\ga}_{I})\\
\label{errofin}
&\leq |||w-w_{h},\ga-\ga_{h}|||\;|||\widetilde{w}-\widetilde{w}^{1}_I,\widetilde{\ga}
-\widetilde{\ga}_{I}|||+\left|a(\bthe-\bthe_h,\nabla\widetilde{w}^{1}_I
+\widetilde{\ga}_{I})+b(\ga-\ga_{h},\widetilde{\ga}_{I})\right|,
\end{align}
where we have used that the bilinear forms are bounded uniformly
in $t$ with respect to the $||| \cdot |||$ norm.
Now, we bound each term on the right hand side above.
For the first term we have, using \eqref{cota1} and \eqref{cota2},
\begin{align*}
|||\widetilde{w}-\widetilde{w}^{1}_I,\widetilde{\ga}-\widetilde{\ga}_{I}|||^2
&\leq C\left(\|\widetilde{w}-\widetilde{w}^{1}_I\|_{2,\O}^2+t^{-2}\|\widetilde{\ga}-\widetilde{\ga}_{I}\|_{0,\O}^2+\|\widetilde{\ga}-\widetilde{\ga}_{I}\|_{1,\O}^2\right)\\
&\leq C\left(  \|\widetilde{w}^1-\widetilde{w}^{1}_I\|_{2,\O}^2+\|\widetilde{w}^2\|_{2,\O}^2+t^{-2}\|\widetilde{\ga}-\widetilde{\ga}_{I}\|_{0,\O}^2+\|\widetilde{\ga}-\widetilde{\ga}_{I}\|_{1,\O}^2\right)\\
&\leq C(h^2+t^2)\|\bthe-\bthe_h\|_{0,\O}^2.
\end{align*}
Therefore
\begin{equation}
\label{ecuacion1}
|||\widetilde{w}-\widetilde{w}^{1}_I,\widetilde{\ga}-\widetilde{\ga}_{I}|||\leq
C (h+t)\|\bthe-\bthe_h\|_{0,\O}.
\end{equation}
For the second term on the right hand of \eqref{errofin},
since $(\widetilde{w}_{I}^{1},\widetilde{\ga}_{I})\in\bfX$,
we have that (see Problems~\ref{P_2} and \ref{P_4}),
\begin{align}
\nonumber
\left|a(\bthe-\bthe_h,\nabla\widetilde{w}^{1}_I+\widetilde{\ga}_{I})
+b(\ga-\ga_{h},\widetilde{\ga}_{I})\right| &=\left|(g,\widetilde{w}_{I}^{1})_{0,\O}
-a(\bthe_h,\nabla\widetilde{w}^{1}_I+\widetilde{\ga}_{I})-b(\ga_{h},\widetilde{\ga}_{I})\right|\\
\nonumber
&\hspace{-4.9cm}=\left| (g,\widetilde{w}_{I}^{1})_{0,\O}-\left<g_{h},\widetilde{w}^{1}_I\right>_{h}
+a_{h}(\bthe_h,\nabla\widetilde{w}^{1}_I+\widetilde{\ga}_{I})+b_{h}(\ga_{h},\widetilde{\ga}_{I})
-a(\bthe_h,\nabla\widetilde{w}^{1}_I+\widetilde{\ga}_{I})-b(\ga_{h},\widetilde{\ga}_{I})\right|\\
\label{contieneA}
&\hspace{-4.9cm}\le B_{1}+B_{2},
\end{align}
where
$$B_{1}:=\left\vert\left(g,\widetilde{w}_{I}^1\right)_{0,\O}-\left<g_{h},\widetilde{w}^{1}_I\right>_{h}\right\vert$$
and
$$B_{2}:=\left\vert a_{h}(\bthe_h,\nabla\widetilde{w}_{I}^1+\widetilde{\ga}_{I})
-a(\bthe_h,\nabla\widetilde{w}_{I}^1+\widetilde{\ga}_{I})+b_{h}(\ga_{h},\widetilde{\ga}_{I})
-b(\ga_{h}\widetilde{\ga}_{I})\right\vert.$$
We now bound $B_{1}$ and $B_{2}$ uniformly in $t$.

We begin with the term $B_{1}$. First adding and subtracting
$\widetilde{w}^{1}$ we have
\begin{align}
B_{1}&\le \left|(g,\widetilde{w}^{1}_I-\widetilde{w}^{1})_{0,\O}\right|
+\left|(g,\widetilde{w}^{1})_{0,\O}-\left<g_{h},\widetilde{w}^{1}\right>_{h}\right|
+\left|\left<g_{h},\widetilde{w}^{1}_I-\widetilde{w}^{1}\right>_{h}\right|\nonumber\\
&\le h^{2}\|g\|_{0,\O}|\widetilde{w}^{1}|_{2,\O}
+\left|(g,\widetilde{w}^{1})_{0,\O}-\left<g_{h},\widetilde{w}^{1}\right>_{h}\right|,\label{conti}
\end{align}
where we have used the Cauchy-Schwarz inequality and Proposition~\ref{estima3}
to bound the first term; note moreover that the last term on the right hand
side above vanish as a consequence of \eqref{gh} and the definition of $\widetilde{w}_{I}^{1}$:
\begin{align*}
\left|\left<g_{h},\widetilde{w}_{I}^{1}-\widetilde{w}^{1}\right>_{h}\right|
=\left|\sum_{E\in\CT_h}\left(\bar{g}_{E}\sum_{i=1}^{N_E}(\widetilde{w}_{I}^{1}-\widetilde{w}^{1})(\vv_i)
\omega_E^{i}\right)\right|=0.
\end{align*}
Now, we bound the second term on the right hand side of
\eqref{conti} and we follow similar steps as
in Lemma~\ref{cotaa} to derive \eqref{cotag}. In fact,
using the definition \eqref{gh},
and adding and subtracting $g_h$ we rewrite the term as follows
\begin{align*}
\left|(g,\widetilde{w}^{1})_{0,\O}-\left<g_{h},\widetilde{w}^{1}\right>_{h}\right|
&=\left|\sum_{E\in\CT_h}\int_{E}g\widetilde{w}^{1}
-\sum_{E\in\CT_h}\left(\bar{g}_{E}\sum_{i=1}^{N_E}\widetilde{w}^{1}(\vv_i)\omega_E^{i}\right)\right|\\
&\hspace{-1cm}\leq \left| \sum_{E\in \CT_h}\int_{E}\bar{g}_{E}\widetilde{w}^{1}
-\sum_{E\in\CT_h}\left(\bar{g}_{E}\sum_{i=1}^{N_E}\widetilde{w}^{1}(\vv_i)\omega_E^{i}
\right)\right|
+\sum_{E\in\CT_h}\|g-\bar{g}_{E}\|_{0,E}\|\widetilde{w}^{1}-p\|_{0,E},
\end{align*}
for any $p\in \bbP_{0}(E)$. Now, taking $p$ as in
Proposition~\ref{estima1} and using that
$g|_{E}\in H^{1}(E)$ and \cite[Lemma 4.3.8]{BS-2008}.
we have that
\begin{equation}\label{number1}
\begin{split}
\left|(g,\widetilde{w}^{1})_{0,\O}-\left<g_{h},\widetilde{w}^{1}\right>_{h}\right|\le&
Ch^2\Vert g\Vert_{1,h,\O}\|\widetilde{w}^{1}\|_{1,\O}
+\left|\sum_{E\in\CT_h}\int_{E}\bar{g}_{E}\widetilde{w}^{1}
-\sum_{E\in\CT_h}\left(\bar{g}_{E}\sum_{i=1}^{N_E}\widetilde{w}^{1}(\vv_i)\omega_E^{i}
\right)\right|\\
=&B_{1,1}+B_{1,2}.
\end{split}
\end{equation}
In what follows
we will manipulate the terms $B_{1,2}$:
adding and subtracting $p_{1}\in \bbP_{1}(E)$,
and the fact that \eqref{gh} is exact for linear functions, we have
\begin{equation}\label{LR:1}
\begin{aligned}
B_{1,2}\le&\left|\sum_{E\in\CT_h}\int_{E}\bar{g}_{E}(\widetilde{w}^{1}-p_{1})\right|
+\left|\sum_{E\in\CT_h}\left(\bar{g}_{E}\left(\sum_{i=1}^{N_E}
(\widetilde{w}^{1}-p_{1})(\vv_i)\omega_E^{i}\right)\right)\right|\\
&\leq \|g\|_{0,\O}\left(\sum_{E\in\CT_h}\|\widetilde{w}^{1}-p_{1}\|_{0,E}^{2}\right)^{1/2}
+\|g\|_{0,\O}\left(\sum_{E\in\CT_h}h_{E}^{2}\|\widetilde{w}^{1}-p_{1}\|_{\infty,E}^{2}\right)^{1/2}.
\end{aligned}
\end{equation}
%%%%%%%%%%%%%%%%%%%%%%
% CHANGED THIS, SEEMED WRONG TO ME!
%%%%%%%%%%%%%%%%%%%%%%
%Now, for any $E\in\CT_h$, let $\bx_{E}$ be the center of the ball
%with respect to which $E$ is starred. Let $T$ be the
%smallest triangle containing the ball
%of center $\bx_{E}$ and radius $h_E$. Then, by Calderon extension
%theorem, there exists $\bar{w}^1\in H^2(T)$
%such that $\bar{w}^1|_{E}=\widetilde{w}^{1}$ and
%$\|\bar{w}^1\|_{2,T}\leq C \|\widetilde{w}^{1}\|_{2,E}$.
%Moreover, there exists $\bar{p}_1\in \bbP_{1}(T)$,
%such that  $\bar{p}_1|_{E}=p_1$. Hence, by a scaling argument in the
%reference triangular element $\widehat{T}$, allow us to conclude that
%%
%\begin{align*}
%\|\widetilde{w}^{1}-p_{1}\|_{\infty,E}\leq&\|\bar{w}^1-\bar{p}_{1}\|_{\infty,T}
%= \|\hat{w}^1-\hat{p}_{1}\|_{\infty,\widehat{T}}\\
%\leq&C\|\hat{w}^1-\hat{p}_{1}\|_{0,\widehat{T}}\leq
%Ch_{T}^{-1}\|\bar{w}^1-\bar{p}_{1}\|_{0,T}\leq
%Ch_{T}\|\bar{w}^1\|_{2,T}\leq Ch_{E}\|\widetilde{w}^{1}\|_{2,E},
%\end{align*}
%where we have used that $\|\bar{w}^1-\bar{p}_{1}\|_{0,T}\leq
%Ch_{T}^2\|\bar{w}^1\|_{2,T}$ with $\bar{p}_{1}$ as in
%Proposition~\ref{estima1}. Moreover, we have that
%$$\|\widetilde{w}^{1}-p_{1}\|_{0,E}\le \|\bar{w}^1-\bar{p}_{1}\|_{0,T}\leq
%Ch_{T}^2\|\bar{w}^1\|_{2,T}\le Ch_{E}^2\|\widetilde{w}^{1}\|_{2,E}$$
%%%%%%%%%%%%%%%%%%%%%%%%%
%Now, in particular, taking $p_{1}\in \bbP_{1}(E)$
%such that Proposition~\ref{estima1} holds with respect to $\widetilde{w}^{1}$,
%we have that
By polynomial approximation results on star-shaped polygons we now have
\begin{equation}\label{LR:2}
\begin{aligned}
%& \|\widetilde{w}^{1}-p_{1}\|_{0,E} \le  C h_{E}^2 |\widetilde{w}^{1}|_{2,E},\\
& \|\widetilde{w}^{1}-p_{1}\|_{\infty,E}  \le C h_{E}|\widetilde{w}^{1}|_{2,E}.
\end{aligned}
\end{equation}
In fact, bound can be derived, for instance, using the following brief guidelines.
Let $B$ be the ball with the same center appearing in ${\bf A}_2$, but radius $h_E$.
It clearly holds $E \subset B$. One can then extend the function $\widetilde{w}^{1}$
to a function (still denoted by $\widetilde{w}^{1}$) in $H^2(B)$ with a uniform bound
$\| \widetilde{w}^{1} \|_{2,B} \le C \| \widetilde{w}^{1} \|_{2,E}$ (see for instance \cite{Stein},
where we use also that due to ${\bf A}_2$ all the elements $E$
of the mesh family are uniformly Lipshitz continuous).
Then, the result follows from the analogous known result
on balls and some very simple calculations.

Hence, using the fact that $\|\widetilde{w}^{1}-p_{1}\|_{0,E}\le
h_E\|\widetilde{w}^{1}-p_{1}\|_{\infty,E}$,
from \eqref{LR:1} and \eqref{LR:2}, we obtain
\begin{equation}\label{number2}
B_{1,2}\le Ch^2\|g\|_{0,\O} |\widetilde{w}^{1}|_{2,\O}.
\end{equation}
Finally, from \eqref{conti}, \eqref{number1} and \eqref{number2}
we have the following bound for the term $B_{1}$:
\begin{align*}
B_{1}\le& Ch^2\|g\|_{1,h,\O}\|\widetilde{w}^{1}\|_{2,\O}\le Ch^2\|g\|_{1,h,\O}\|\bthe-\bthe_h\|_{0,\O}
\le C(h+t)h\|g\|_{1,h,\O}\|\bthe-\bthe_h\|_{0,\O}.
\end{align*}

Now, we bound the term $B_{2}$ in \eqref{contieneA}.
First, we consider $(w_{\Pi},\ga_{\Pi},\ga_0)\in[\LO]^{5}$
such that $w_{\Pi}|_{E}\in\bbP_{2}(E)$,
$\ga_{\Pi}|_{E}\in[\bbP_{1}(E)]^2$ and $\ga_{0}|_{E}\in[\bbP_{0}(E)]^2$
and define $\bthe_{\Pi}:=\nabla w_{\Pi}+\ga_{\Pi}$.
Moreover, we consider $(\widetilde{w}^{1}_{\Pi},\widetilde{\ga}_{\Pi},\widetilde{\ga}_0)\in[\LO]^{5}$
such that $\widetilde{w}^{1}_{\Pi}|_{E}\in\bbP_{2}(E)$,
$\widetilde{\ga}_{\Pi}|_{E}\in[\bbP_{1}(E)]^2$ and $\widetilde{\ga}_{0}|_{E}\in[\bbP_{0}(E)]^2$.
Thus, using the consistency property we rewrite the term as follows
\begin{align*}
B_{2}&= \Big\vert\sum_{E\in \CT_{h}}\left(a_{h}^{E}(\bthe_h,\nabla\widetilde{w}^{1}_I
+\widetilde{\ga}_{I}-(\nabla\widetilde{w}^{1}_{\Pi}+\widetilde{\ga}_{\Pi}))
+a_{h}^{E}(\bthe_h,\nabla\widetilde{w}^{1}_{\Pi}+\widetilde{\ga}_{\Pi})\right)\\
&\hspace{0.5cm} - \sum_{E\in \CT_{h}}\left(a^E(\bthe_h,\nabla\widetilde{w}^{1}_I
+\widetilde{\ga}_{I}-(\nabla\widetilde{w}^{1}_{\Pi}+\widetilde{\ga}_{\Pi}))
+a^E(\bthe_h,\nabla\widetilde{w}^{1}_{\Pi}+\widetilde{\ga}_{\Pi})\right)\\
&\hspace{0.5cm} +\sum_{E\in \CT_{h}}\left(b_{h}^E(\ga_{h},\widetilde{\ga}_{I}
-\widetilde{\ga}_{0})+b_{h}^E(\ga_{h},\widetilde{\ga}_{0})-b^{E}(\ga_{h},\widetilde{\ga}_{I}
-\widetilde{\ga}_{0})-b^E(\ga_{h},\widetilde{\ga}_{0})\right)\Big\vert\\
&=\Big\vert\sum_{E\in \CT_{h}}\left(a_{h}^{E}(\bthe_h-\bthe_{\Pi},\nabla\widetilde{w}^{1}_I
+\widetilde{\ga}_{I}-(\nabla\widetilde{w}^{1}_{\Pi}+\widetilde{\ga}_{\Pi}))
-a^{E}(\bthe_h-\bthe_{\Pi},\nabla\widetilde{w}^{1}_I+\widetilde{\ga}_{I}
-(\nabla\widetilde{w}^{1}_{\Pi}+\widetilde{\ga}_{\Pi}))\right)\\
&\hspace{0.5cm} +\sum_{E\in \CT_{h}}\left(b_{h}^E(\ga_{h}-\ga_{0},\widetilde{\ga}_{I}
-\widetilde{\ga}_{0})-b^{E}(\ga_{h}-\ga_{0},\widetilde{\ga}_{I}-\widetilde{\ga}_{0})\right)\Big\vert.
\end{align*}
Therefore, we have
\begin{align*}
B_{2}&\leq C\left(\|\bthe_h-\bthe_{\Pi}\|_{1,h,\O}+t^{-1}\|\ga_{h}-\ga_{0}\|_{0,\O}\right)\times\\
&\qquad\left(\sum_{E\in \CT_{h}}\|\nabla\widetilde{w}^{1}_I-\nabla\widetilde{w}^{1}_{\Pi}\|_{1,E}^{2}
+\|\widetilde{\ga}_{I}-\widetilde{\ga}_{\Pi}\|_{1,E}^{2}+t^{-2}\|\widetilde{\ga}_{I}
-\widetilde{\ga}_{0}\|_{0,E}^{2}\right)^{1/2}\\
&\leq C\left(\|\bthe_h-\bthe_{\Pi}\|_{1,h,\O}+t^{-1}\|\ga_{h}
-\ga_{0}\|_{0,\O}\right)\left(h|\widetilde{w}^{1}|_{3,\O}+ht^{-1}|\widetilde{\ga}|_{1,\O}
+|\widetilde{\ga}|_{1,\O}\right),
\end{align*}
where we have added and subtracted $\nabla\widetilde{w}^{1}$ and $\widetilde{\ga}$
and then we have used Propositions~\ref{estima3}, \ref{estima1} and \ref{estima4},
respectively. Finally, using \eqref{regular} and the triangular inequality we have
\begin{equation*}\label{ecuacion2}
B_{2}\leq C(h+t)\|\bthe-\bthe_h\|_{0,\O}\left(|||w-w_{h},\ga-\ga_{h}|||
+\|\bthe-\bthe_{\Pi}\|_{1,h,\O}+t^{-1}\|\ga-\ga_{0}\|_{0,\O}\right).
\end{equation*}
Hence, \eqref{frgt1} follows from \eqref{errofin}, combining the estimate
\eqref{ecuacion1}, with the above bounds for $B_{1}$ and $B_{2}$
and the definition of $\bthe$. In fact, we obtain that
\begin{align*}
&\|\bthe-\bthe_h\|_{0,\O}\leq C(h+t)\left(|||w-w_{h},\ga-\ga_{h}|||+h\Vert g\Vert_{1,h,\O}
+\|\nabla w-\nabla w_{\Pi}\|_{1,h,\O}\right.\\
&\hspace{6.7cm}\left.+\|\ga-\ga_{\Pi}\|_{1,h,\O}+t^{-1}\|\ga-\ga_{0}\|_{0,\O}\right).
\end{align*}

Finally, bound \eqref{frgt2} follows from the Poincar\'e inequality
and the triangular inequality we have that
\begin{equation*}\label{hgjrt}
\|w-w_h\|_{1,\O}\leq C\|\nabla w-\nabla w_h\|_{0,\O}=C\|\bthe-\ga-(\bthe_h-\ga_h)\|_{0,\O}
\leq C(\|\bthe-\bthe_h\|_{0,\O}+\|\ga-\ga_h\|_{0,\O}).
\end{equation*}
The proof is complete.
\end{proof}

Finally, we obtain the following result.
\begin{corollary}\label{picodepedra}
Assume that the hypotheses of Theorem~\ref{aproximation2} hold.
Moreover, assume that the domain $\O$ be either regular,
or piecewise regular and convex, that $g\in H^{1}(E)$ for all $E\in\CT_h$
and that $\G_c=\G$.
Then, there exists $C>0$
independent of $h$, $g$ and $t$ such that
\begin{align*}
\|\bthe-\bthe_h\|_{0,\O}+ \|w-w_h\|_{1,\O}&\leq
C(h+t)h\left(t^{-1}|\ga|_{1,\O}+|\bthe|_{2,\O}+|w|_{3,\O}+\Vert g\Vert_{1,h,\O}\right).\\
\end{align*}
 \end{corollary}
\begin{proof} 
The proof follows directly from Proposition~\ref{aproximation3},
combining Theorem~\ref{aproximation2}, Propositions~\ref{estima1}
and the fact that $\|\ga-\ga_h\|_{0,\O}\leq t|||w-w_{h},\ga-\ga_h|||$.
\end{proof}

\begin{remark}
We note that the shear strain variable in the present paper is given
by $\ga=\nabla w-\bthe$ and it is related with the usual
scaled shear strain used in other Reissner-Mindlin contributions
in the literature as follows $Q=t^{-2}\ga$. Since $t^{-1}\ga=tQ$
is a quantity that is known to be uniformly bounded
for clamped boundary conditions
in the correct Sobolev norms (see, e.g \cite{AF,BFS91}).
Therefore, the factors $t^{-1}$ appearing in Theorem~\ref{aproximation2}
and Corollary~\ref{picodepedra} are not a source of locking.
\end{remark}

\begin{remark}\label{rem:reg}
We note that in our convergence results, in order to obtain the
full convergence rate in $h$ (independently of the thickness $t$)
we need $|w|_{3,\Omega}$ to be bounded uniformly in $t$.
We observe that such condition can be
achieved on a smooth domain $\Omega$ and regular data (see \cite[Remark 1]{ABFM2007}).
On the other hand, on less regular domains $\Omega$,
even in the presence of regular data, the regularity for $w$ is not
assured due to the presence of layers at the boundaries of the plate and singularities at corners.
Such limitation of the above theoretical analysis is related to the adopted formulation
and is, somehow, the drawback related to the advantage of having
a method with $C^1$ deflections, that is therefore able to give
(at the limit for vanishing thickness) a Kirchhoff conforming solution.
We finally note that, in practice, this kind of difficulty can be
effectively dealt with by an ad-hoc refinement of the mesh
near the boundaries or corners of the plate; an example is shown later in Section \ref{sec:L-shaped}.
\end{remark}

%%%%%%%%%%%%%%%%%%%%%%%%%%%%%%%%%%%%%%%%%%%%%%%%%%%%%%%%%%%%%
%
%   TEMPORANEO COMMENTO DELLA SEZIONE TEST-NUMERICI
%
%%%%%%%%%%%%%%%%%%%%%%%%%%%%%%%%%%%%%%%%%%%%%%%%%%%%%%%%%%%%%

\section{Numerical results}
\label{SEC:NUMER}

We report in this section some numerical examples which
have allowed us to assess the theoretical results proved above.
We have implemented in a MATLAB code our method
on arbitrary polygonal meshes, by following the ideas proposed
in \cite{BBMR2014}. To complete the choice of the VEM,
we have to fix the bilinear forms $S^E(\cdot,\cdot)$
and $S_0^E(\cdot,\cdot)$ satisfying
\eqref{stabilS} and \eqref{stabilS0}, respectively.
Proceeding as in \cite{BBMR2014},
% --------
%For each element $E\in\CT_h$ with vertices
%$\vv_1,\ldots,\vv_{N_{E}}$ and edges $e_{1},\ldots,e_{N_{E}}$,
%let $\{\bvarphi_{1},\ldots,\bvarphi_{N_{\boldsymbol{E}}},\ldots,\bvarphi_{3N_{E}}\}$
%be the dual basis of $\bVE$ associated with the degrees of freedom
%$\mathcal{V}_E^h$ (the values of $\bvarphi_{i}$ at the vertices)
%and $\CE_E^h$ (see \eqref{freedom2}); We assume that all the degrees
%of freedom are scaled in such a way that the associated dual basis
%$\{\bvarphi_{i}\}_{i=1}^{3N_{E}}$ scales uniformly as follows 
%$$\|\bvarphi_{i}\|_{\infty,E}\approx 1\qquad \forall i=1,\ldots,3N_{E}.$$
% -------
a natural choice for $S^E(\cdot,\cdot)$  is given by
\begin{equation*}
S^E(\ga_{h},\bta_h):= \sigma_{E}\left(\sum_{i=1}^{2 N_{E}}\ga_{h}
(\vv_{i})\bta_h(\vv_i)+\sum_{j=1}^{N_{E}}\left(\dfrac{1}{|e_{j}|}
\int_{e_{j}}\ga_{h}\cdot\boldsymbol{t} \right)\left(\dfrac{1}{|e_{j}|}
\int_{e_{j}}\bta_h\cdot\boldsymbol{t}\right)\right),\quad \ga_{h},\bta_{h}\in\bVE,
\end{equation*}
where $\sigma_{E}>0$ is
a multiplicative factor
to take into account the magnitude
of the material parameter, for instance,
in the numerical tests a possible choice could be to set 
$\sigma_{E}>0$ as the mean
value of the eigenvalues of the local matrix
$a^E(\PiE\ga_h,\PiE\bta_h)$.
This ensure that the stabilizing term scales as
$a^E(\bta_h,\bta_h)$.
Now, a choice for $S_{0}^E(\cdot,\cdot)$ is given by
\begin{equation*}
S_{0}^E(\ga_{h},\bta_h):=\frac{\l h_{E}^{2}}{t^2}\left(\sum_{i=1}^{2N_{E}}\ga_{h}
(\vv_{i})\bta_h(\vv_{i})+\sum_{j=1}^{N_{E}}\left(\dfrac{1}{|e_{j}|}
\int_{e_{j}}\ga_{h}\cdot\boldsymbol{t} \right)\left(\dfrac{1}{|e_{j}|}
\int_{e_{j}}\bta_h\cdot\boldsymbol{t}\right)\right),\quad \ga_{h},\bta_{h}\in\bVE.
\end{equation*}
In this case, we have multiplied the stabilizing term by the material/geometric parameter $\l t^{-2}$
to ensure \eqref{stabilS0}. A proof of \eqref{stabilS}-\eqref{stabilS0}
for the above (standard) choices could be derived following the arguments in \cite{BLRXX}.

The choices above are standard in the Virtual Element literature,
and correspond to a scaled identity matrices in the space of the
degrees of freedom values.

%\begin{remark}
%We note that, since the bilinear form $b(\cdot,\cdot)$ is multiplied
%by the physical parameters $\l t^{-2}$, the bilinear forms
%$S_{0}^E(\cdot,\cdot)$ are multiplied by the physical parameter
%$\l t^{-2}$ to ensure \eqref{stabilS0}.  Similarly $S^{E}(\cdot,\cdot)$
%is multiplied by the Young modules $E$ to ensure \eqref{stabilS}.
%\end{remark}}

To test the convergence properties of the method,
we introduce the following discrete $L^{2}$-like norm:
for any sufficiently regular function $\bv$,
$$\|\bv\|_{0,\O}^{2}:=\disp\sum_{E\in \CT_{h}}
\left(|E|\disp\sum_{i=1}^{N_E}\left(\bv(\vv_i)\right)^{2}\right),$$
with $|E|$ being the area of element $E$.
We also define the relative errors in discrete $L^{2}$-like norms (based on the vertex values):
\begin{equation*}
\left(e_{w}\right)^{2} :=
\dfrac{\disp\sum_{E\in \CT_{h}}\left(|E|\disp\sum_{i=1}^{N_E}
\left(w(\vv_i)-w_{h}(\vv_i)\right)^{2}\right)}{\disp\sum_{E\in \CT_{h}}
\left(|E|\disp\sum_{i=1}^{N_E}\left(w(\vv_i)\right)^{2}\right)} ,
% \left(e_{\nabla w}\right)^{2}&:=\dfrac{\|\nabla w-\nabla w_{h}\|_{0,\O}^{2}}{\|\nabla w\|_{0,\O}^{2}}\qquad
% {\rm and }\quad
% \left(e_{\boldsymbol{\theta}}\right)^{2}:=\dfrac{\|\boldsymbol{\theta}-\boldsymbol{\theta}_{h}\|_{0,\O}^{2}}{\|\boldsymbol{\theta}\|_{0,\O}^{2}}.
\end{equation*}
and the obvious analogs for $e_{\nabla w}$ and $e_{\boldsymbol{\theta}}$.
Finally, we introduce the relative error in the energy norm
\begin{align*}
\left(\boldsymbol{\CE}\right)^{2}&:=\dfrac{\mathcal{A}_h((w-w_h,\ga-\ga_h),
(w-w_h,\ga-\ga_h))}{\mathcal{A}_h((w,\ga), (w,\ga))},
\end{align*}
where $\mathcal{A}_h(\cdot,\cdot)$ corresponds to the discrete bilinear form
on the left hand side of Problem~\ref{P_4}.

% --------------------------------------
\subsection{Test 1:}

As a test problem we have taken an isotropic and homogeneous plate $\O:=(0,1)^{2}$,
clamped on the whole boundary, for which the analytical
solution is explicitly known (see \cite{CL}).

Choosing the transversal load $g$ as:
\begin{align*}
g(x,y)&=\dfrac{\mathbb{E}}{12(1-\nu^{2})}\left[12y(y-1)(5x^{2}-5x+1)(2y^{2}(y-1)^{2}+x(x-1)(5y^{2}-5y+1)\right.\\
&\left.+12x(x-1)(5y^{2}-5y+1)(2x^{2}(x-1)^{2}+y(y-1)(5x^{2}-5x+1)\right],
\end{align*}
the exact solution of the problem is given by:
\begin{align*}
w(x,y)&=\dfrac{1}{3}x^{3}(x-1)^{3}y^{3}(y-1)^{3}\\
&-\dfrac{2t^{2}}{5(1-\nu)}\left[y^{3}(y-1)^{3}x(x-1)(5x^{2}-5x+1)+x^{3}(x-1)^{3}y(y-1)(5y^{2}-5y+1)\right],
\end{align*}
\[\boldsymbol{\theta}(x,y)=\left[ \begin{array}{c}
y^{3}(y-1)^{3}x^{2}(x-1)^{2}(2x-1)\\
x^{3}(x-1)^{3}y^{2}(y-1)^{2}(2y-1)
\end{array} \right].
\]
The shear modulus $\l$ is given by $\l:=\dfrac{5\mathbb{E}}{12(1+\nu)}$
(choosing 5/6 as shear correction factor), while
the material constants have been chosen
$\mathbb{E} = 1$ and $\nu=0$.

We have tested the method by using
different values of the plate thickness:
$t=1.0e-01$, $t=1.0e-02$ and $t=1.0e-03$.
Moreover, we have used different families of meshes
(see Figure~\ref{FIG:VM1}):
\begin{itemize}
\item $\CT_h^1$: triangular meshes;
\item $\CT_h^2$: trapezoidal meshes which consist
of partitions of the domain into $N\times N$ congruent
trapezoids, all similar to the trapezoid with
vertices $(0,0)$, $(\dfrac{1}{2},0)$, $(\dfrac{1}{2},\dfrac{2}{3})$ and $(0,\dfrac{1}{3})$; 
\item $\CT_h^3$: triangular meshes, considering
the middle point of each edge as a new degree of
freedom but moved randomly; note that these meshes
contain non-convex elements.
\end{itemize}
The refinement parameter $h$ used to label each mesh is $h=\disp\max_{E\in \CT_{h}} h_{E}$.
 \begin{figure}[H]
\begin{center}
\begin{minipage}{4.3cm}
\centering\includegraphics[height=4.1cm, width=4.1cm]{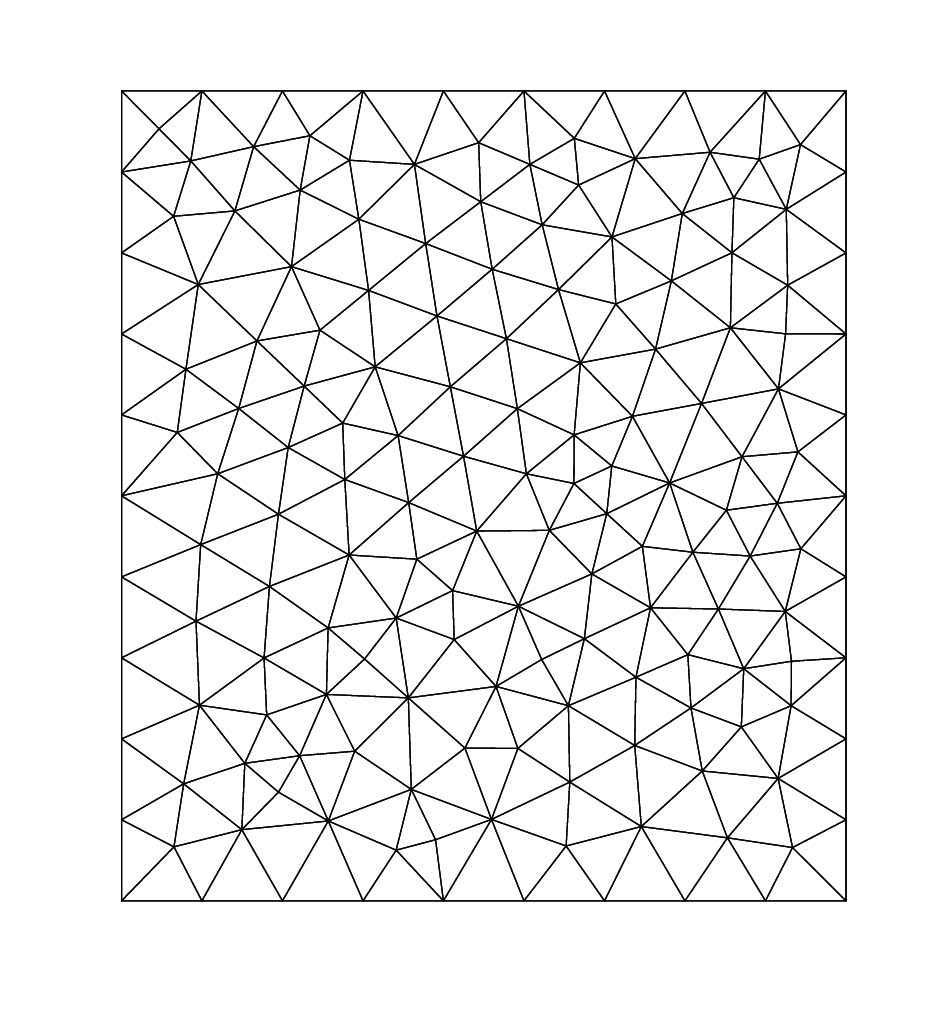}
\end{minipage}
\begin{minipage}{4.3cm}
\centering\includegraphics[height=4.1cm, width=4.1cm]{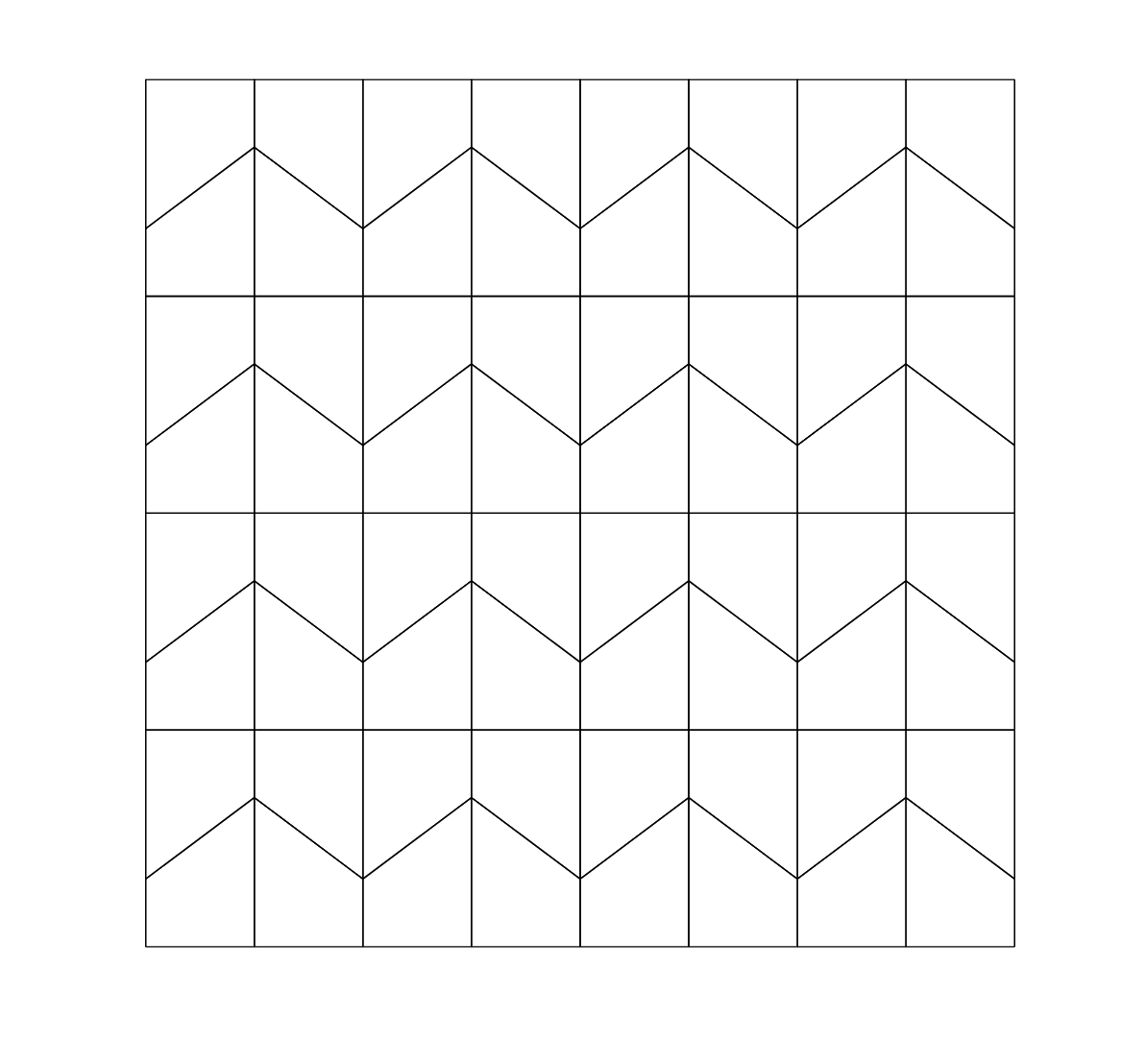}
\end{minipage}
\begin{minipage}{4.3cm}
\centering\includegraphics[height=4.1cm, width=4.1cm]{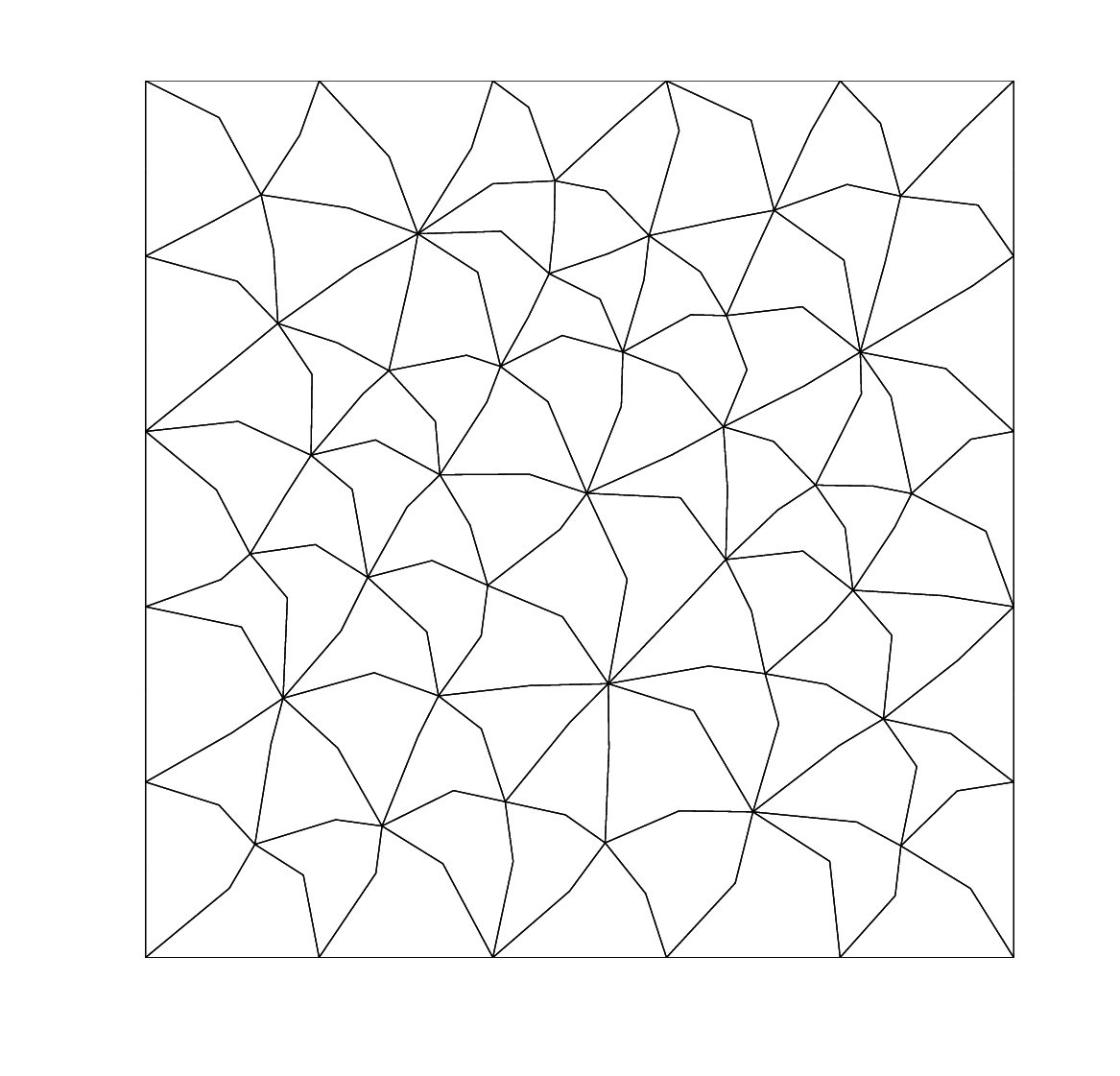}
\end{minipage}
\caption{ Sample meshes: $\CT_h^1$ (left), $\CT_h^2$ (middle) and
$\CT_h^{3}$ (right) with $h=0.1189 $, $h=0.1719$ and $h=0.11078$, respectively. }
 \label{FIG:VM1}
\end{center}
\end{figure}

We report in Table~\ref{TABLA:1}, Table~\ref{TABLA:2} 
and Table~\ref{TABLA:3} the relative errors in
the discrete $L^{2}$-norm of $w$,
$\nabla w$ and $\bthe$, together with the
relative errors in the energy norm,
for each family of meshes and different refinement levels.
We consider different thickness: $t=1.0e-01$, $t=1.0e-02$ and $t=1.0e-03$, respectively.
We also include in these table the experimental rate of convergence.
  
\begin{table}[H]
\begin{center}
\caption{$\CT_h^1$: Computed error in the discrete $L^{2}$-norm
with  $t=1.0e-01$, $t=1.0e-02$ and $t=1.0e-03$, respectively.}
\begin{tabular}{|c|c|c|c|c|c|c|c|c|}
\hline
error    & $h=0.119 $ & $ h=  0.0588$ & $ h=   0.0314$ & $ h=   0.0158 $ & $ h= 0.00827$& 
Order \\
\hline
$e_{w}     $& 1.108e-01 &   2.941e-02  & 7.423e-03  & 1.841e-03  & 4.679e-04&2.06 \\
 $e_{\nabla w}$  & 1.300e-01 &  4.168e-02 &  1.321e-02&   4.230e-03 &  1.502e-03&1.69\\
 $e_{\boldsymbol{\theta}}$   &8.873e-02 &  2.229e-02  & 5.434e-03  & 1.331e-03  & 3.367e-04 & 2.10 \\
 $\boldsymbol{\CE}$   &2.994e-01&   1.278e-01 &  5.395e-02 &  2.276e-02 &  1.070e-02&  1.26\\
 \hline
    \hline 
$e_{w}     $& 1.030e-01  & 2.633e-02&   6.493e-03 &  1.596e-03   &4.041e-04&2.09 \\
 $e_{\nabla w}$  &8.953e-02 &  2.255e-02&   5.477e-03&   1.343e-03 &  3.397e-04&  2.10\\
$e_{\boldsymbol{\theta}}$   &8.925e-02   &2.244e-02  & 5.445e-03  & 1.335e-03 &  3.375e-04 &   2.10\\
$\boldsymbol{\CE}$   & 1.756e-01 &  8.550e-02&   4.199e-02 &  2.020e-02 &  1.029e-02 &1.07\\
\hline
\hline   
$e_{w}     $& 1.030e-01  & 2.631e-02   &6.488e-03 &  1.596e-03 &  4.043e-04 &2.09\\
 $e_{\nabla w}$  &8.926e-02   &2.245e-02 &  5.452e-03  & 1.339e-03  & 3.387e-04 &2.10 \\
$e_{\boldsymbol{\theta}}$   &8.926e-02 &  2.245e-02 &  5.452e-03 &  1.338e-03 &  3.387e-04&   2.10\\
$\boldsymbol{\CE}$   &1.733e-01&   8.475e-02 &  4.168e-02 &  1.998e-02 &  1.014e-02& 1.07\\
\hline
\end{tabular}
\label{TABLA:1}
\end{center}
\end{table}

\begin{table}[H]
\begin{center}
\caption{$\CT_h^2$: Computed error in the discrete $L^{2}$-norm
with  $t=1.0e-01$, $t=1.0e-02$ and $t=1.0e-03$, respectively.}
\begin{tabular}{|c|c|c|c|c|c|c|c|c|c|c|}
\hline
error    & $h=0.172$ &$h=  0.0859$ & $h=  0.0430$  &$h=  0.0215$ & $h= 0.0122 $    & 
Order \\
\hline
$e_{w}     $ & 3.890e-01  & 1.110e-01&   2.958e-02&   7.612e-03 &  1.868e-03 &1.91 \\   
$e_{\nabla w}$  & 4.147e-01&   1.370e-01&   4.381e-02&   1.405e-02 & 4.653e-03   &1.61   \\   
$e_{\boldsymbol{\theta}}$  &3.647e-01   &9.742e-02   &2.480e-02 &  6.247e-03&1.535e-03 &1.96\\
$\boldsymbol{\CE}$   &6.267e-01 &  3.104e-01 &  1.345e-01 &  5.812e-02 &2.527e-02 &1.16\\
\hline
\hline   
 $e_{w}     $ &3.796e-01  & 1.037e-01 &  2.651e-02 &  6.656e-03& 1.608e-03&1.96\\   
$e_{\nabla w}$  &3.663e-01 &  9.831e-02  & 2.501e-02  & 6.278e-03&  1.534e-03& 1.96  \\   
$e_{\boldsymbol{\theta}}$   &3.659e-01   &9.802e-02   &2.490e-02 &  6.245e-03 &  1.525e-03&1.96   \\
$\boldsymbol{\CE}$   &4.199e-01  & 1.675e-01  & 7.480e-02&   3.619e-02& 1.825e-02 &1.12\\
\hline
\hline      
$e_{w}     $ & 3.795e-01&   1.036e-01 &  2.650e-02 &  6.662e-03 &1.612e-03 &1.95    \\   
$e_{\nabla w}$  &3.659e-01 &  9.803e-02  & 2.491e-02  & 6.254e-03& 1.528e-03&1.96    \\   
$e_{\boldsymbol{\theta}}$   &3.659e-01   &9.803e-02   &2.491e-02  & 6.253e-03& 1.528e-03&1.96    \\
$\boldsymbol{\CE}$   & 4.154e-01 &  1.6465e-01 &  7.360e-02 &  3.548e-02 & 1.772e-02 &1.12 \\
\hline
\end{tabular}
\label{TABLA:2}
\end{center}
\end{table}

\begin{table}[H]
\begin{center}
\caption{$\CT_h^3$: Computed error in the discrete $L^{2}$-norm
with  $t=1.0e-01$, $t=1.0e-02$ and $t=1.0e-03$, respectively.}
\begin{tabular}{|c|c|c|c|c|c|c|c|c|c|c|}
\hline
error    & $h=0.111 $ &$h= 0.0594   $ & $h= 0.0294$  &$h=0.0157$      &$h=0.00791$ & 
Order \\
\hline
$e_{w}     $ & 2.848e-01 &  8.769e-02  & 2.421e-02&   6.267e-03&   1.569e-03 & 2.06    \\   
$e_{\nabla w}$  &  3.162e-01 &  1.202e-01 &  4.091e-02&   1.338e-02&   4.447e-03& 1.70  \\   
$e_{\boldsymbol{\theta}}$  &2.566e-01  & 7.465e-02   &1.914e-02   &4.740e-03 &  1.161e-03& 2.14 \\
$\boldsymbol{\CE}$   &6.371e-01&   3.653e-01 &  1.690e-01 &  7.647e-02 &  3.470e-02&1.17 \\
\hline
\hline   
 $e_{w}     $ &2.660e-01 &  7.407e-02 &  1.9065e-02  & 4.761e-03&   1.161e-03&2.15         \\   
$e_{\nabla w}$  &2.568e-01 &  7.495e-02  & 1.912e-02  & 4.778e-03 &  1.167e-03&   2.14     \\   
$e_{\boldsymbol{\theta}}$   &2.562e-01   &7.460e-02   &1.898e-02  & 4.736e-03  & 1.155e-03 & 2.15    \\
$\boldsymbol{\CE}$   &3.793e-01  & 1.975e-01  & 9.555e-02 &  4.832e-02  & 2.365e-02& 1.10\\
\hline
\hline      
$e_{w}     $ & 2.664e-01  & 7.458e-02 &  1.902e-02 &  4.742e-03 &  1.161e-03 & 2.18    \\   
$e_{\nabla w}$  &2.561e-01&   7.468e-02  & 1.907e-02  & 4.726e-03  & 1.161e-03&2.17       \\   
$e_{\boldsymbol{\theta}}$   & 2.560e-01  & 7.468e-02  & 1.907e-02  & 4.726e-03  & 1.161e-03&  2.17      \\
$\boldsymbol{\CE}$   & 3.744e-01 &  1.904e-01 &  9.447e-02 &  4.757e-02 &  2.343e-02 &     1.11 \\
\hline
\end{tabular}
\label{TABLA:3}
\end{center}
\end{table}

It can be seen from Tables~\ref{TABLA:1}, \ref{TABLA:2} and \ref{TABLA:3}
that the theoretical predictions of Section~\ref{SEC:approximation}
are confirmed. In particular, we can appreciate a rate of convergence $O(h)$ for the energy norm $\boldsymbol{\CE}$, that is equivalent to the $||| \cdot |||$ norm. This holds for all the considered meshes and thicknesses, 
thus also underlying the locking free nature of the scheme.
Moreover, for sufficiently small $t$ we also observe
a clear rate of convergence $O(h^2)$ for
for $e_{w}$, $e_{\nabla w}$ and $e_{\boldsymbol{\theta}}$, in accordance with Corollary \ref{picodepedra}.

% -------------------------------------
\subsection{Test 2:}

As a second test, we investigate more in depth
the locking-free character of the method, and also
take the occasion for a comparison with the limit Kirchhoff model.
It is well known (see \cite{BF}) that when $t$ goes to zero the
solution of the Reissner-Mindlin
model converges
to an identical Kirchhoff-Love solution:
Find $w_{0}\in \HdO$ such that
\begin{equation}
\label{kk}
\dfrac{\mathbb{E}}{12(1-\nu^{2})}\Delta^{2}w_{0}=g ,
\end{equation}
with the corresponding boundary conditions.

We have considered a rectangular plate
$\O:=(0,a)\times(0,b)$, simply supported on the whole boundary,
and we have chosen the transversal load $g$ as
\begin{equation*}
g(x,y)=\sin\left(\dfrac{\pi}{a} x\right)\sin\left(\dfrac{\pi}{b}y\right).
\end{equation*}
Then, the analytical solution $w_{0}$
of problem~\eqref{kk} is given by
\begin{align*}
w_{0}(x,y)&=\dfrac{12(1-\nu^{2})}{\mathbb{E}}\left(\pi^{4}\left(\dfrac{1}{a^{2}}+\dfrac{1}{b^{2}}\right)^{2}\right)^{-1}
\sin\left(\dfrac{\pi}{a} x\right)\sin\left(\dfrac{\pi}{b}y\right).
\end{align*}

The material constants have been chosen
$\mathbb{E} = 1$ and $\nu=0.3$. Moreover,
we have taken $a=1$ and $b=2$,
and we have used three different
families of meshes (see Figure~\ref{FIG:VM2}): 
\begin{itemize}
\item $\CT_h^1$: triangular meshes;
\item $\CT_h^4$: hexagonal meshes;
\item $\CT_h^5$: Voronoi polygonal meshes.
\end{itemize}

\begin{figure}[H]
\begin{center}
\begin{minipage}{4.3cm}
\centering\includegraphics[height=6cm, width=3.5cm]{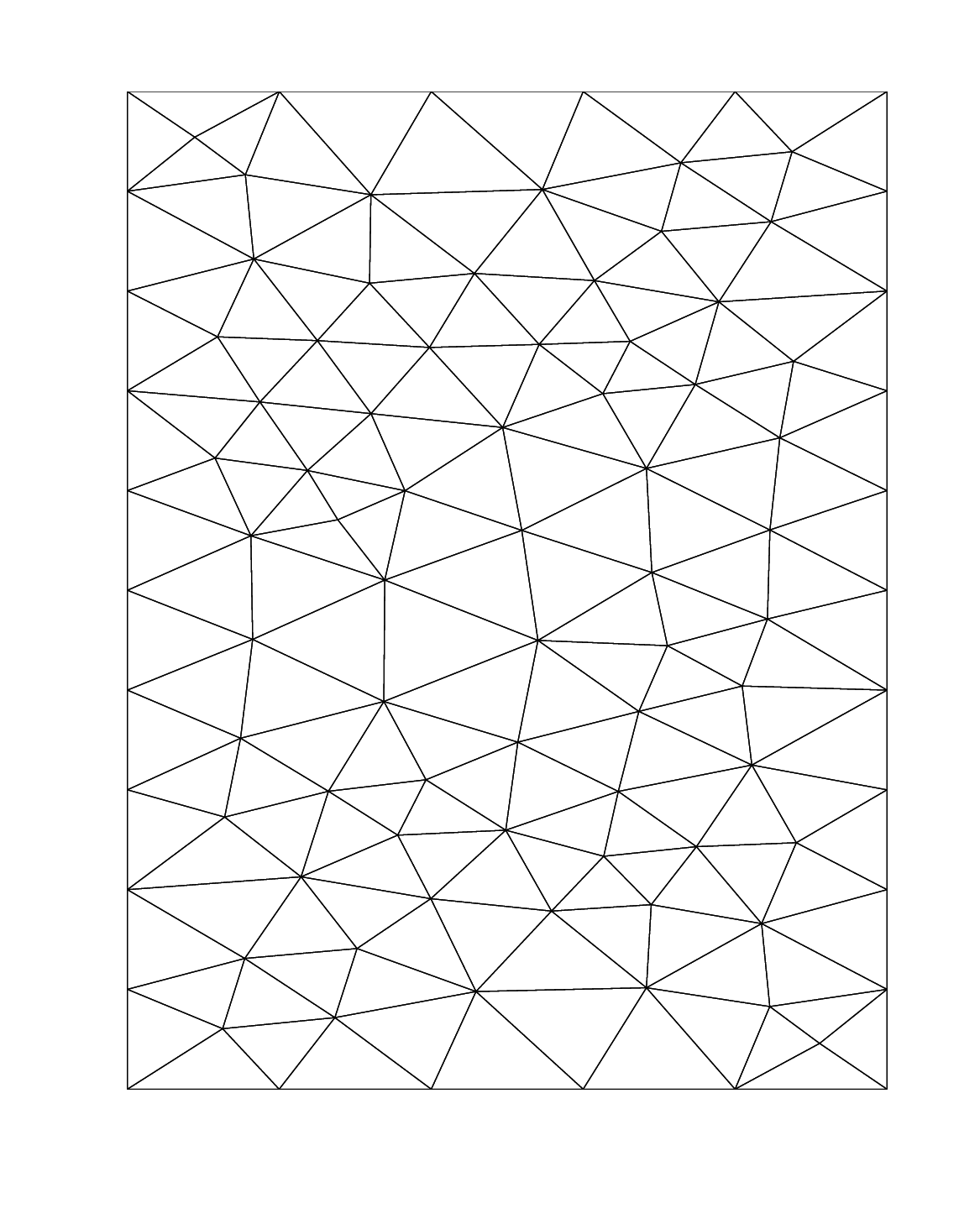}
\end{minipage}
\begin{minipage}{4.3cm}
\centering\includegraphics[height=6cm, width=3.5cm]{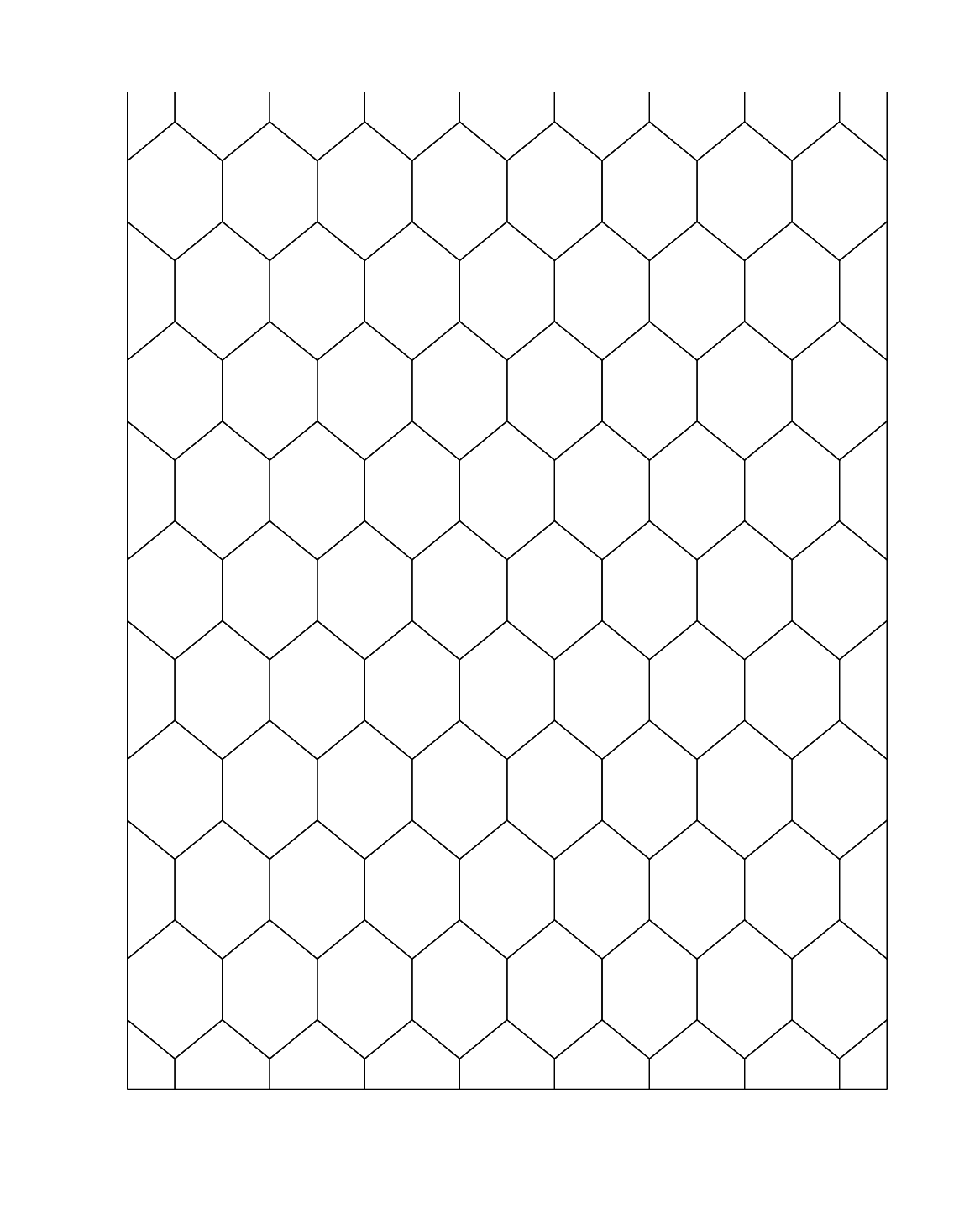}
\end{minipage}
\begin{minipage}{4.3cm}
\centering\includegraphics[height=6cm, width=3.5cm]{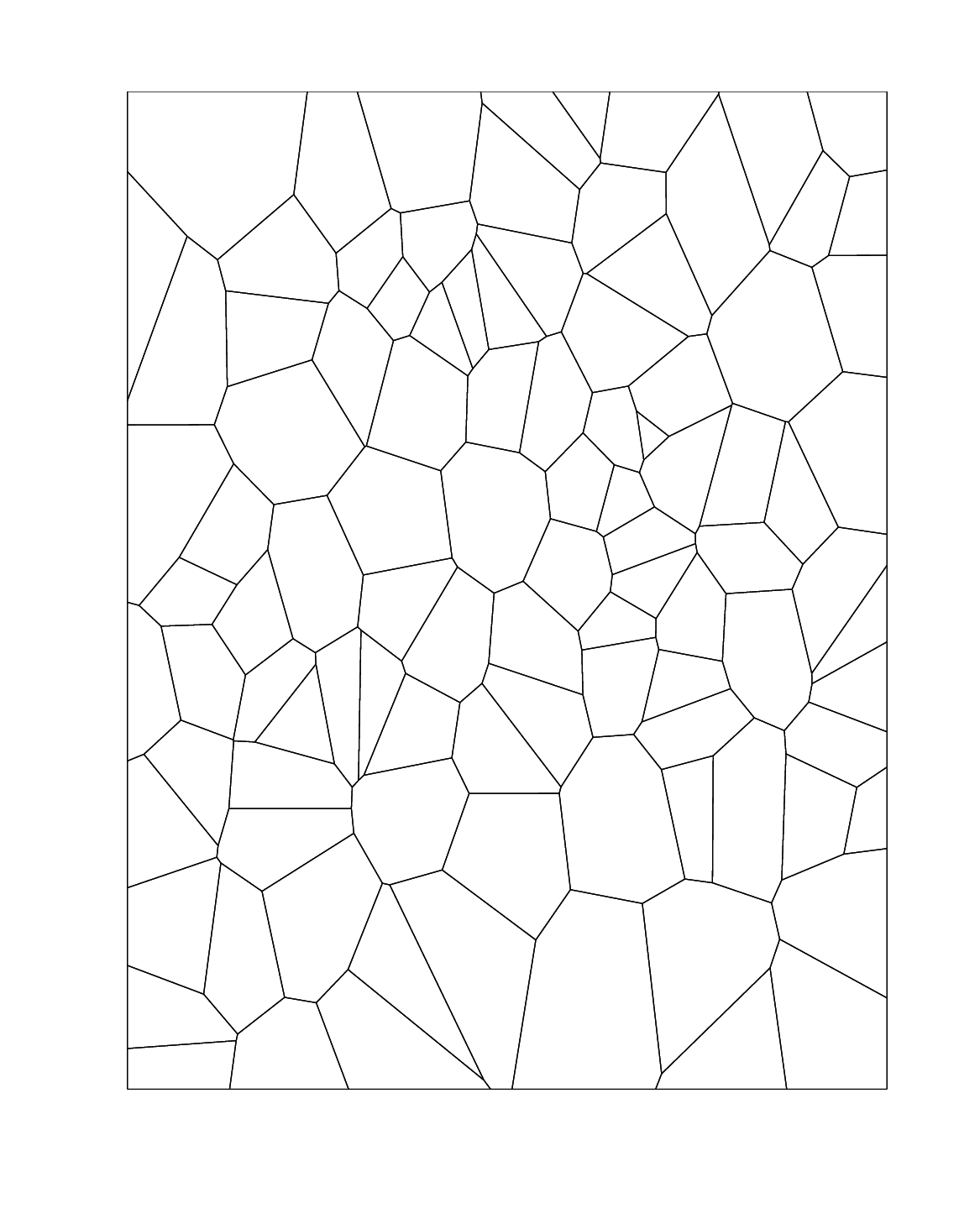}
\end{minipage}
\caption{ Sample meshes: $\CT_h^1$ (left), $\CT_h^4$ (middle) and
$\CT_h^{5}$ (right). }
\label{FIG:VM2}
\end{center}
\end{figure}

Tables~\ref{TABLA:6}, \ref{TABLA:9} and \ref{TABLA:10}
show an analysis for various thicknesses
in order to assess the locking-free nature of the proposed method.
It is shown the relative errors in
the discrete $L^{2}$-norm which are obtained by comparing the numerical solution with
the Kirchhoff-Love plate solution $w_{0}$
for each family of meshes and different refinement levels and
considering different thickness: $t=1.0e-01$,
$t=1.0e-02$, $t=1.0e-03$, $t=1.0e-04$ and $t=1.0e-05$, respectively.

It can be clearly seen from these tables that the proposed method
is locking-free. The lack of error reduction for finer values of $h$,
which can be observed for the case $t=1.0e-01$,
is clearly due to the fact that the model error is dominating the discretization error in those cases.

\begin{table}[H]
\begin{center}
\caption{Computed error in $e_{w}$ by $\CT_h^1$.}
\begin{tabular}{|c|l|l|l|l|l|l|l|l|}
\hline
  $t\backslash h$& 2.449e-01 &  1.271e-01 &  6.469e-02 &  3.241e-02 &  1.617e-02\\ 
\hline
   1.0e-01 &  8.609e-03  & 4.090e-02 &  6.039e-02  & 7.108e-02 &  7.562e-02\\
   1.0e-02  & 4.687e-02  & 1.034e-02  & 1.798e-03  & 7.665e-04  & 2.012e-03\\
   1.0e-03  & 4.730e-02  & 1.096e-02   &2.719e-03  & 6.668e-04   &1.385e-04\\
   1.0e-04   &4.730e-02 &  1.097e-02  & 2.728e-03&   6.823e-04   &1.664e-04\\
   1.0e-05 & 4.730e-02  &1.097e-02 &  2.728e-03 &  6.825e-04&     1.666e-04\\
      \hline 
    \end{tabular}
\label{TABLA:6}
\end{center}
\end{table}
\begin{table}[H]
\begin{center}
\caption{Computed error in $e_{w}$ by $\CT_h^4$.}
\begin{tabular}{|c|l|l|l|l|l|l|l|l|}
\hline
  $t\backslash h$& 2.781e-01 &  1.309e-01  &  6.730e-02 &  4.443e-02  &  3.316e-02\\
 \hline
 1.0e-01 &   8.394e-02 &  5.936e-02   &  6.256e-02&    6.715e-02&   6.910e-02\\
1.0e-02 &   4.979e-02   &1.018e-02      &4.022e-03  & 2.476e-03  &  2.095e-03\\
1.0e-03 &   4.943e-02 &  9.575e-03 &    3.145e-03 &    1.333e-03  &  7.381e-04\\
1.0e-04 &   4.942e-02   &9.569e-03   &3.136e-03   & 1.321e-03  &  7.228e-04\\
1.0e-05 &    4.942e-02  & 9.569e-03 & 3.136e-03  &  1.321e-03 &   7.227e-04  \\
      \hline
    \end{tabular}
\label{TABLA:9}
\end{center}
\end{table}

\begin{table}[H]
\begin{center}
\caption{Computed error in $e_{w}$ by $\CT_h^5$.}
\begin{tabular}{|c|l|l|l|l|l|l|l|l|}
\hline
  $t\backslash h$&4.592e-01  & 2.348e-01 &  1.294e-01 &  8.174e-02 &  5.507e-02 \\
 \hline
 1.0e-01 &   2.762e-02  & 4.055e-02  & 4.618e-02 &  6.510e-02  & 6.982e-02\\
1.0e-02 &    1.270e-02 &  3.454e-03  & 7.218e-04  & 1.141e-03  & 1.393e-03\\
1.0e-03 &    1.277e-02 &  3.004e-03 &  3.816e-04 &  6.483e-05 &  4.532e-05\\
1.0e-04 &    1.277e-02 &  2.999e-03 &  3.876e-04  & 6.248e-05 &  3.257e-05\\
1.0e-05 &    1.277e-02 &2.999e-03 &  3.874e-04 &6.215e-05 & 3.193e-05\\
      \hline
    \end{tabular}
\label{TABLA:10}
\end{center}
\end{table}

% -------------------------------------
\subsection{Test 3:}
\label{sec:L-shaped}

In this numerical example we test the properties
of the proposed method on an L-shaped plate:
$\O:=(0,1)\times(0,1)\setminus[0.5,1)\times[0.5,1)$. 

The plate is clamped on the edges $\{0\}\times[0, 1]$,
$\{1\}\times[0, 1/2]$, $[0,1]\times\{0\}$, $[0,1/2]\times\{1\}$,
and free on the remaining boundary and subjected to the constant 
transversal load $g=1$ (constant on the whole domain)
and we take the material constants as
$\mathbb{E} = 1$ and $\nu=0$, with
shear correction factor $k=5/6$. The thickness is set as $t=1.0e-01$.

We consider two families of meshes
(see Figure~\ref{FIG:VM123}): 
\begin{itemize}
% \item $\CT_h^1$: triangular meshes;
\item $\CT_h^6$: a sequence of uniform squares meshes;
the first one is constructed by subdividing into $8 \times 8$
squares each of the three squares composing $\Omega$
(see upper left picture in Figure ~\ref{FIG:VM123}),
up to the last one that is associated to an analogous $40 \times 40$ subdivision.
\item $\CT_h^7$: polygonal meshes obtained following a very
simple procedure that refines the mesh only around the re-entrant
corner, starting from an initial uniform square mesh (that corresponds to the coarser mesh in $\CT_h^6$).
It consists of splitting each element which has the free corner  $(1/2,1/2)$
as a vertex into four quadrilaterals by connecting the
barycenter of the element with the midpoint of each edge.
Notice that although this process is initiated with a mesh of squares,
the successively created meshes will contain other kind of convex polygons
as can be seen in Figure~\ref{FIG:VM123}.
\end{itemize}
The main purpose of this test is to validate the use of refined meshes as
a tool to handle solutions with corner singularities (therefore,
in particular, not in $H^3(\Omega)$) and therefore overcome the
theoretical limitation underlined in Remark \ref{rem:reg}.
Moreover, the possibility of using polygonal meshes makes
such refinement construction simpler, as shown by the example above;
boundary layer treatment would obviously follow an analogous approach.

In Table~\ref{TABLA:12} we report the value of the transversal displacement
of the plate at the free corner $(1/2,1/2)$ and the number of degrees of freedom associated to each mesh.
Since we have no exact solution for this problem, the last line in the table shows the reference values obtained with
a very fine triangular mesh with the finite element method introduced and analyzed in \cite{DL}. 
We note that the family of meshes $\CT_h^7$, being refined only around the corner,
cannot obtain convergence as the error generated far from the corner would eventually
dominate. Nevertheless, family $\CT_h^7$ fits completely into the scope of the present test:
comparing the displacement values obtained by $\CT_h^7$ with those of the uniform
meshes one can clearly appreciate the efficiency of the proposed corner refinement in handling the singularity.
% \begin{table}[H]
% \begin{center}
% \caption{\textcolor{blue}{Test with an L-shaped plate. Number of degrees of freedom,
% transversal displacement of the free corner  $(1/2,1/2)$ for the
% considered meshes, errors and a reference value for thickness $t=1.0e-01$.}}
% \begin{tabular}{|c|c|c|c|}
% \hline
% & 1541 & 0.019534273067153 & 0.000206293342366\\
% & 2345 & 0.019530983666828 &0.000209582742691\\
% & 5765 & 0.019575890897794 &0.000164675511725\\
%   &8885 & 0.019604412633453 &0.000136153776066\\ 
% $\CT_h^6$  &19625 &0.019651903522007 &0.000088662887512\\
% & 22277 &0.019658446807712 &0.000082119601807\\
%  &34565 &0.019678557085842 &0.000062009323677\\
% % & 53705 &0.019694839267450 &0.00241875687\\
% \hline
% ref. & 181603&0.019740566409519 & \\
% \hline
% \hline
% & 1541   & 0.019534273067153  &0.000206293342366 \\
%   & 1628 &  0.019651252437530  &0.000089313971989 \\ 
% $\CT_h^7$  &1715  & 0.019700456265659 &0.000040110143860\\
% & 1802 &  0.019721788645938 &0.000018777763581\\
%  &1889 &  0.019730835088572 &0.000009731320947 \\
% \hline
% ref. & 181603&0.019740566409519 &\\
% \hline
% 
% \end{tabular}
% \label{TABLA:12}
% \end{center}
% \end{table}

\begin{table}[H]
\begin{center}
\caption{Test with an L-shaped plate. Number of degrees of freedom,
transversal displacement of the free corner  $(1/2,1/2)$ for the
considered meshes, errors and a reference value for thickness $t=1.0e-01$.}
\begin{tabular}{|c|c|c|c|}
\hline
& 1541 & 0.01953427 & 2.0629e-04\\
& 2345 & 0.01953098 &2.0958e-04\\
& 5765 & 0.01957589 &1.6468e-04\\
  &8885 &0.01960441 &1.3615e-04\\ 
$\CT_h^6$  &19625 &0.01965190 & 8.8663e-05\\
& 22277 &0.01965845 &8.2120e-05\\
 &34565 &0.01967856 &6.2009e-05\\
% & 53705 &0.019694839267450 &0.00241875687\\
\hline
ref. & 181603&0.01974057 & \\
\hline
\hline
& 1541   & 0.01953427  &2.0629e-04 \\
  & 1628 & 0.01965125  &8.9314e-05 \\ 
$\CT_h^7$  &1715  & 0.01970046 &4.0110e-05\\
& 1802 &  0.01972179 &1.8778e-05\\
 &1889 &  0.01973084 &9.7313e-06 \\
\hline
ref. & 181603&0.01974057 &\\
\hline

\end{tabular}
\label{TABLA:12}
\end{center}
\end{table}

\begin{figure}[H]
\begin{center}
\begin{minipage}{5.3cm}
\centering\includegraphics[height=5.1cm, width=5.1cm]{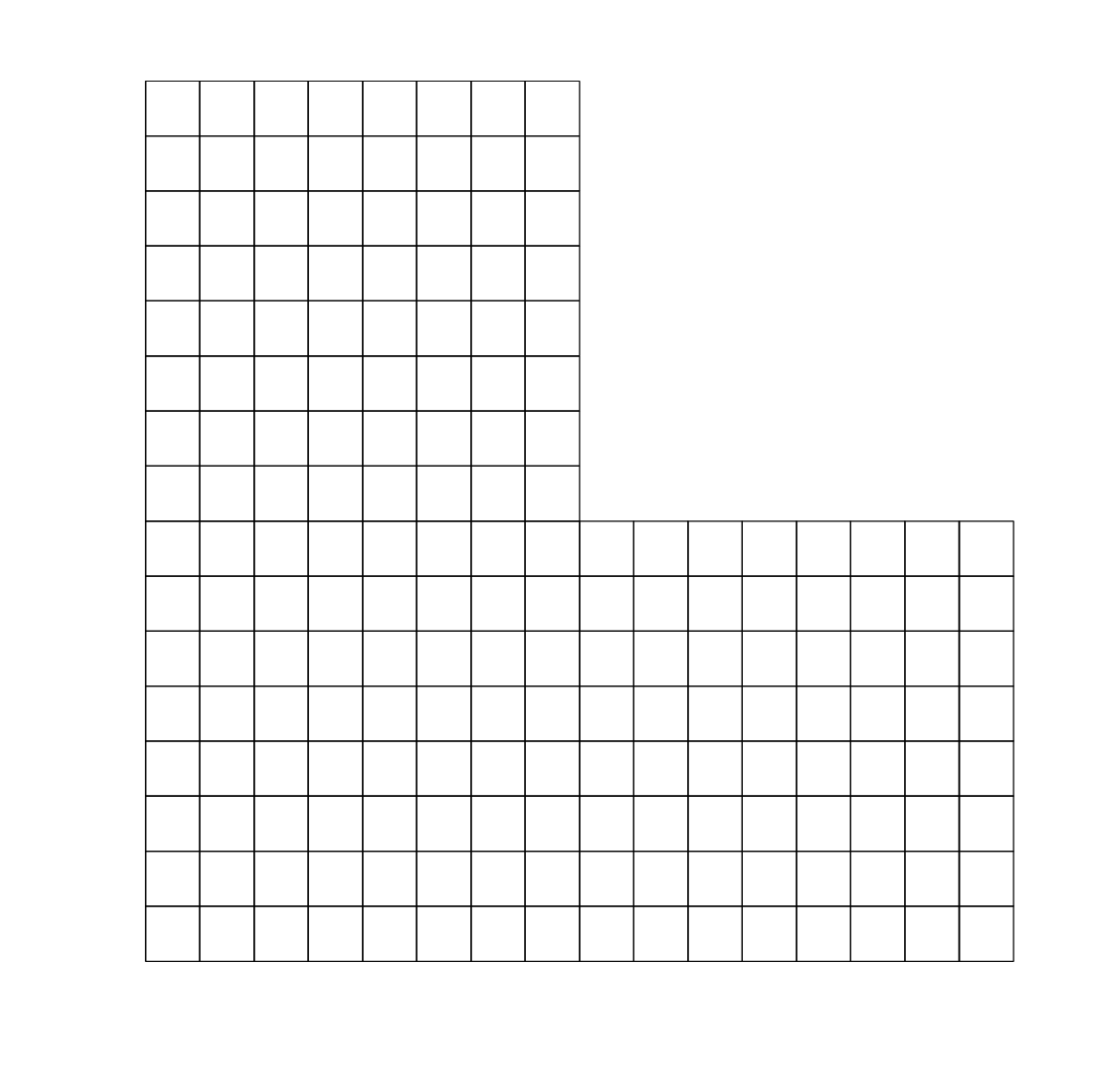}
\centering{Initial mesh (both for $\CT_h^6$ and $\CT_h^7$).}
\end{minipage}
\begin{minipage}{5.3cm}
\centering\includegraphics[height=5.1cm, width=5.1cm]{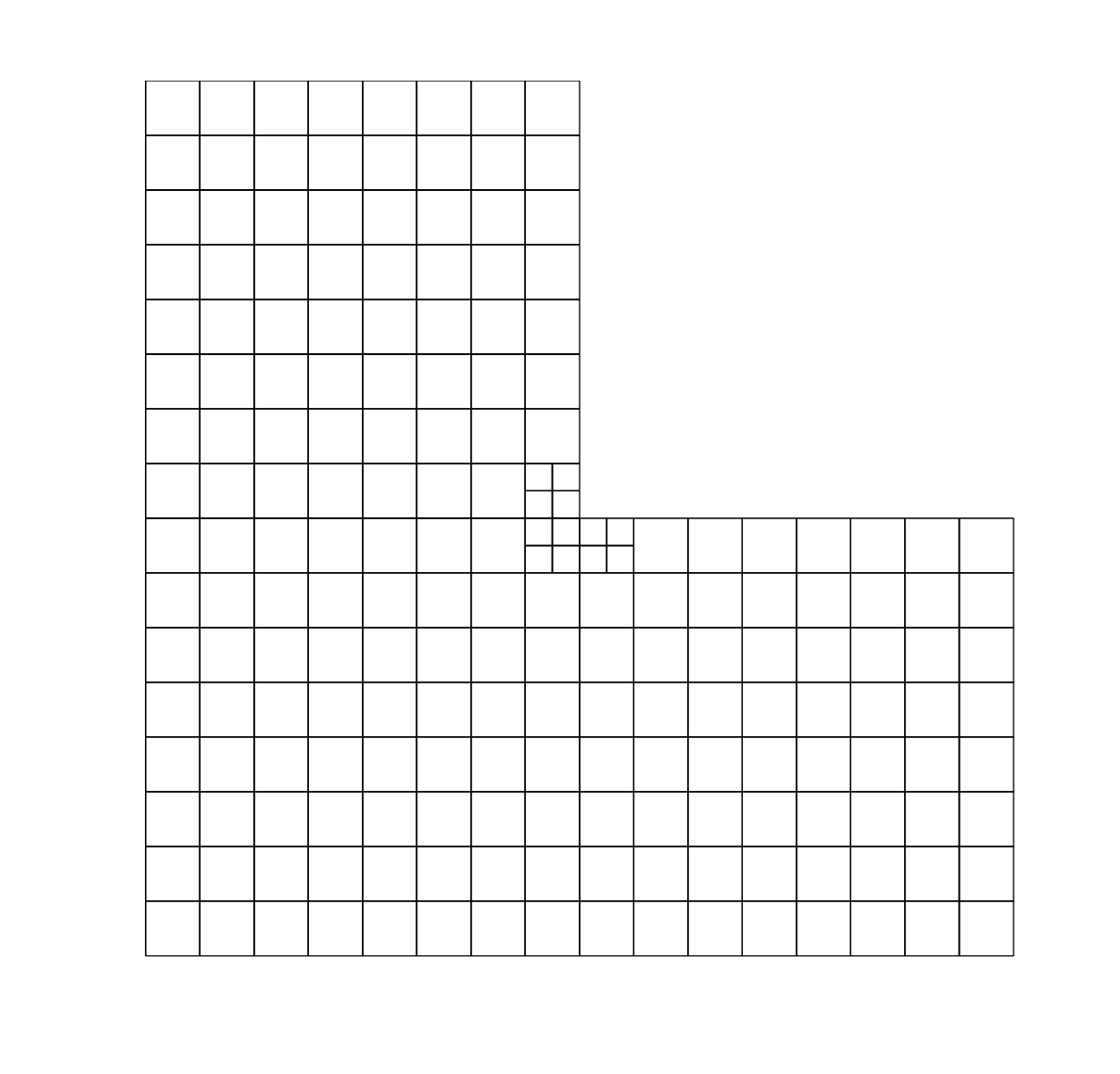}
\centering{Mesh 2 of $\CT_h^7$.}
\end{minipage}\\
\begin{minipage}{5.3cm}
\centering\includegraphics[height=5.1cm, width=5.1cm]{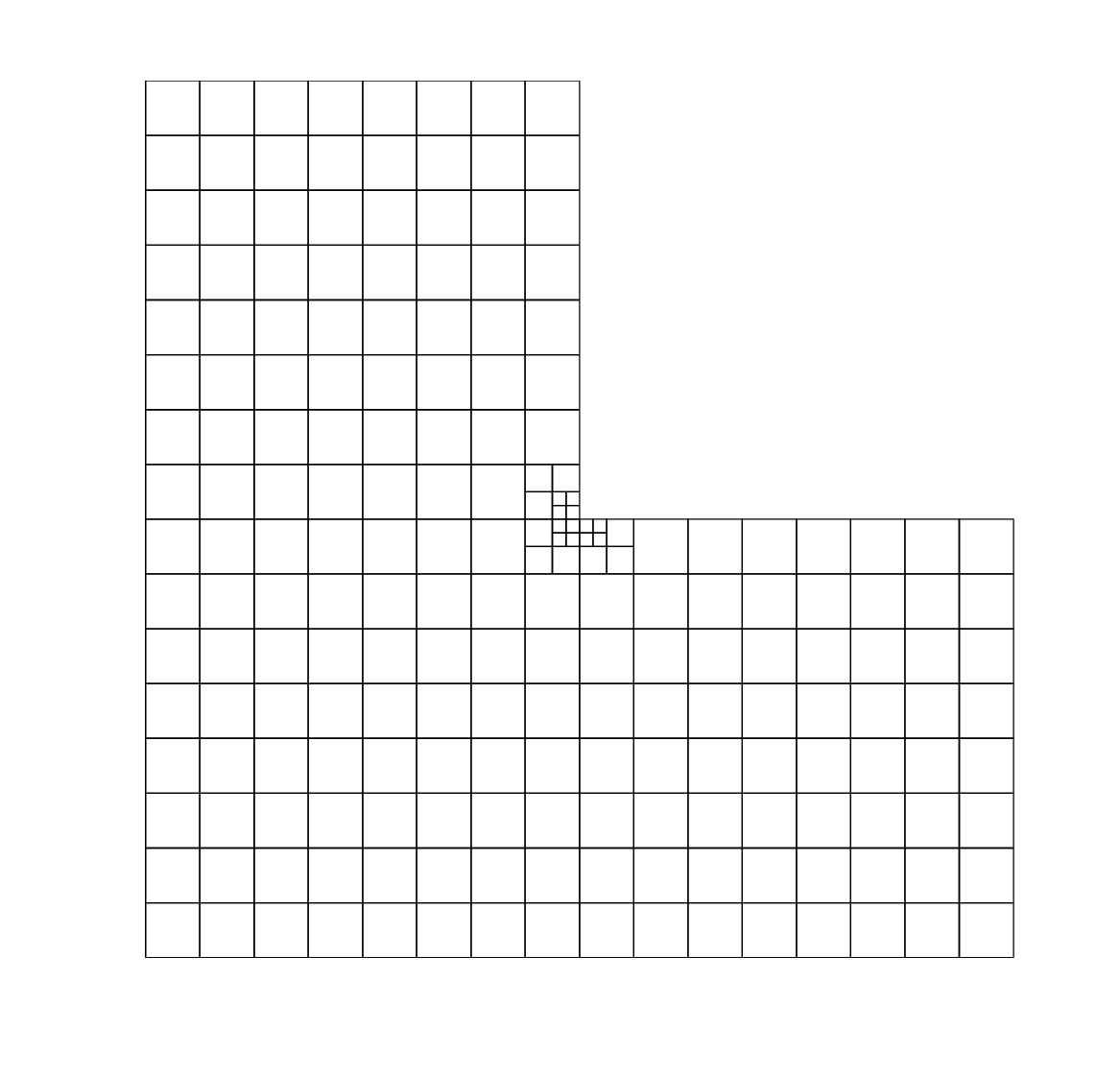}
\centering{Mesh 3 of $\CT_h^7$.}
\end{minipage}
\begin{minipage}{5.3cm}
\centering\includegraphics[height=5.1cm, width=5.1cm]{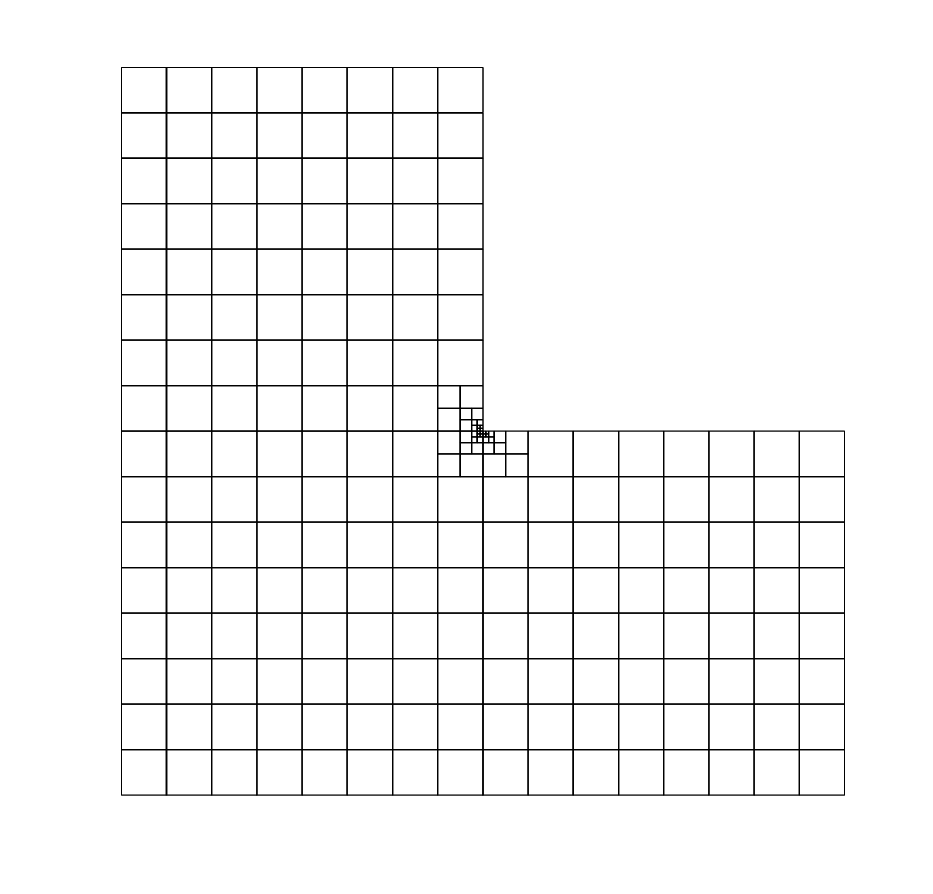}
\centering{Mesh 6 of $\CT_h^7$.}
\end{minipage}
\caption{Sample meshes on L-shaped domain.}
 \label{FIG:VM123}
\end{center}
\end{figure}

\section*{Acknowledgement}
The authors are deeply grateful Prof. Rodolfo Rodr\'iguez
(Universidad de Concepci\'on) for the fruitful discussions.

\bibliographystyle{amsplain}

\begin{thebibliography}{24}



\bibitem{AABMR13} 
\textsc{B. Ahmad, A. Alsaedi, F. Brezzi, L.D. Marini and A. Russo},
\textit{Equivalent projectors for virtual element methods},
Comput. Math. Appl., {\bf 66}, (2013), pp.  376--391.

\bibitem{ABMVsinum14} 
\textsc{P.F. Antonietti, L. Beir\~ao da Veiga, D. Mora and M. Verani},
\textit{A stream virtual element formulation of the Stokes problem on polygonal meshes},
SIAM J. Numer. Anal., {\bf 52(1)}, (2014), pp. 386--404.

\bibitem{ABSVsinum16} 
\textsc{P.F. Antonietti, L. Beir\~ao da Veiga, S. Scacchi and M. Verani},
\textit{A $C^1$ virtual element method for the Cahn--Hilliard equation with polygonal meshes},
SIAM J. Numer. Anal., {\bf 54(1)}, (2016), pp. 36--56.

\bibitem{AHSV} 
\textsc{P.F. Antonietti, P. Houston, X. Hu, M. Sarti and M. Verani},
\textit{Multigrid algorithms for $hp$-version interior penalty discontinuous
Galerkin methods on polygonal and polyhedral meshes},
Calcolo, DOI: 10.1007/s10092-017-0223-6 (2017).


\bibitem{ABFM2007} 
\textsc{D.N. Arnold, F. Brezzi, R.S. Falk and L.D. Marini},
\textit{Locking-free Reissner-Mindlin elements without reduced integration},
Comput. Methods Appl. Mech. Engrg., {\bf 196}, (2007) pp. 3660--3671.


\bibitem{AF} 
\textsc{D.N. Arnold and R.S. Falk},
\textit{A uniformly accurate finite element method for the Reissner-Mindlin plate},
SIAM J. Numer. Anal., {\bf 26}, (1989), pp. 1276--1290.

\bibitem{ALM15} 
\textsc{B. Ayuso de Dios, K. Lipnikov and G. Manzini},
\textit{The nonconforming virtual element method},
ESAIM Math. Model. Numer. Anal., {\bf 50(3)}, (2016), pp. 879--904.


\bibitem{BBCMMR2013} 
\textsc{L. Beir\~ao da Veiga, F. Brezzi, A. Cangiani, G. Manzini, L.D. Marini and A. Russo},
\textit{Basic principles of virtual element methods},
Math. Models Methods Appl. Sci., {\bf 23}, (2013), pp. 199--214.

%\bibitem{BBM} 
%\textsc{L. Beir\~ao~da Veiga, F. Brezzi and L.D. Marini},
%\textit{Virtual elements for linear elasticity problems},
%SIAM J. Numer. Anal., {\bf 51} (2013), pp. 794--812.


\bibitem{BBMR2014} 
\textsc{L. Beir\~ao da Veiga, F. Brezzi, L.D. Marini and A. Russo},
\textit{The hitchhiker's guide to the virtual element method},
Math. Models Methods Appl. Sci., {\bf 24}, (2014), pp. 1541--1573.

%\bibitem{BBMR2015} 
%\textsc{L. Beir\~ao da Veiga, F. Brezzi, L.~D. Marini and A. Russo},
%\textit{Mixed virtual element methods for general second order
%elliptic problems on polygonal meshes},
%ESAIM Math. Model. Numer. Anal., DOI: http://dx.doi.org/10.1051/m2an/2015067
%(2016).

\bibitem{BHKLNRS} 
\textsc{L. Beir\~ao da Veiga, T.J.R. Hughes, J. Kiendl,
C. Lovadina, J. Niiranen, A. Reali and H. Speleers},
\textit{A locking-free model for Reissner-Mindlin plates: analysis and isogeometric
implementation via NURBS and triangular NURPS},
Math. Models Methods Appl. Sci., {\bf 25(8)}, (2015), pp. 1519--1551.


\bibitem{BLMbook2014} 
\textsc{L. Beir\~ao da Veiga, K. Lipnikov and G. Manzini},
\textit{The Mimetic Finite Difference Method for Elliptic Problems},
Springer, MS\&A, vol. {\bf 11}, 2014.

\bibitem{BLM2015} 
\textsc{L. Beir\~ao da Veiga, C. Lovadina and D. Mora},
\textit{A virtual element method for elastic and inelastic problems on polytope meshes},
Comput. Methods Appl. Mech. Engrg., {\bf 295}, (2015) pp. 327--346.

\bibitem{BLRXX}
\textsc{L. Beir\~ao da Veiga, C. Lovadina and A. Russo},
\textit{Stability analysis for the virtual element method},
Math. Models Methods Appl. Sci., DOI: 10.1142/S021820251750052X (2017).


\bibitem{BLV} 
\textsc{L. Beir\~ao da Veiga, C. Lovadina  and G. Vacca},
\textit{Divergence free virtual elements for the Stokes problem on polygonal meshes},
ESAIM Math. Model. Numer. Anal., {\bf 51(2)}, (2017) pp. 509--535.


%\bibitem{BM13} 
%\textsc{L. Beir\~ao da Veiga and G. Manzini},
%\textit{A virtual element method with arbitrary regularity},
%IMA J. Numer. Anal., {\bf 34}, (2014), pp. 759--781.

\bibitem{BMRR} 
\textsc{L. Beir\~ao da Veiga, D. Mora, G. Rivera and R. Rodr\'iguez},
\textit{A virtual element method for the acoustic vibration problem},
Numer. Math., {\bf 136(3)}, (2017) pp. 725--763.


\bibitem{BBBPS2016} 
\textsc{M.F. Benedetto, S. Berrone, A. Borio, S. Pieraccini and S. Scial\`o},
\textit{A hybrid mortar virtual element method for discrete fracture network simulations},
J. Comput. Phys., {\bf 306}, (2016), pp. 148--166.


\bibitem{BBPS2014} 
\textsc{M.F. Benedetto, S. Berrone, S. Pieraccini and S. Scial\`o},
\textit{The virtual element method for discrete fracture network simulations},
Comput. Methods Appl. Mech. Engrg., {\bf 280}, (2014), pp. 135--156.


\bibitem{bertoluzza}
\textsc{S. Bertoluzza},
\textit{Substructuring preconditioners for the three fields domain decomposition method}, 
Math. Comp. {\bf 73}, (2004), pp. 659--689. 


\bibitem{BS-2008} 
\textsc{S.C. Brenner and R.~L. Scott},
\textit{The Mathematical Theory of Finite Element Methods},
Springer, New York, 2008.

\bibitem{ultimo} 
\textsc{F. Brezzi, R.S. Falk and L.D. Marini},
\textit{Basic principles of mixed virtual element methods},
ESAIM Math. Model. Numer. Anal., {\bf 48}, (2014), pp. 1227--1240.


\bibitem{BF} 
\textsc{F. Brezzi and M. Fortin},
\textit{Mixed and Hybrid Finite Element Methods},
Springer-Verlag, New York, (1991).

\bibitem{BFS91} 
\textsc{F. Brezzi, M. Fortin and R. Stenberg},
\textit{Error analysis of mixed-interpolated elements for Reissner-Mindlin plates},
Math. Models Meth. Appl. Sci., {\bf 1}, (1991) pp. 125--151.


\bibitem{BM12} 
\textsc{F. Brezzi and L.D. Marini},
\textit{Virtual elements for plate bending problems},
Comput. Methods Appl. Mech. Engrg., {\bf 253}, (2012), pp. 455--462.


\bibitem{Bergh-Lofstrom}
\textsc{J. Bergh and J. L\"ofstr\"om},
\textit{Interpolation spaces. An introduction.}
Springer-Verlag, Berlin--New York, 1976.


\bibitem{CG2016} 
\textsc{E. Caceres and G.N. Gatica},
\textit{A mixed virtual element method for the pseudostress-velocity
formulation of the Stokes problem},
IMA J. Numer. Anal., {\bf 37(1)}, (2017) pp. 296--331.



\bibitem{CGH14} 
\textsc{A. Cangiani, E.H.  Georgoulis and P. Houston},
\textit{$hp$-version discontinuous Galerkin methods on polygonal and polyhedral meshes},
Math. Models Methods Appl. Sci., {\bf24(10)}, (2014), pp. 2009--2041.


\bibitem{CL} 
\textsc{C. Chinosi and C. Lovadina},
\textit{Numerical analysis of some mixed finite
element methods for Reissner-Mindlin plates},
Comput. Mech., {\bf 16(1)}, (1995), pp. 36--44.

\bibitem{ChM-camwa} 
\textsc{C. Chinosi and L.D. Marini},
\textit{Virtual element method for fourth order problems: $L^2$-estimates},
Comput. Math. Appl., {\bf 72(8)}, (2016), pp. 1959--1967.


\bibitem{ciarlet}
\textsc{P.G. Ciarlet},
\textit{The Finite Element Method for Elliptic Problems},
SIAM, 2002.

\bibitem{ciarlet2} 
\textsc{P.G. Ciarlet},
\textit{Interpolation error estimates for the
 reduced Hsieh--Clough--Tocher triangle},
Math. Comp., {\bf 32(142)}, (1978), pp. 335--344.


\bibitem{DPECMAME2015} 
\textsc{D. Di Pietro and A. Ern},
\textit{A hybrid high-order locking-free method for linear elasticity on general meshes},
Comput. Methods Appl. Mech. Eng., {\bf283}, (2015), pp. 1--21.

%\bibitem{DPECRAS2015} 
%\textsc{D. Di Pietro and A. Ern},
%\textit{Hybrid high-order methods for
%variable-diffusion problems on general meshes},
%C. R. Acad. Sci., Paris I, {\bf353}(1), (2015), pp. 31--34.


\bibitem{DL} 
\textsc{R. Dur\'an and E. Liberman},
\textit{On mixed finite elements methods for the Reissner-Mindlin plate model},
Math. Comput., {\bf 58}, (1992), pp. 561--573. 



\bibitem{EOB2013} 
\textsc{R. Echter, B. Oesterle  and M. Bischoff},
\textit{A hierarchic family of isogeometric shell finite elements},
Comput. Methods Appl. Mech. Engrg., {\bf 254}, (2013) pp. 170--180.

\bibitem{falk} 
\textsc{R. Falk},
\textit{Finite elements for the Reissner-Mindlin plate},
D. Boffi and L. Gastaldi, editors, Mixed finite
elements, compatibility conditions, and applications,
Springer, Berlin, 2008, pp. 195--232.



\bibitem{Paulino-VEM}
\textsc{A.L. Gain, C. Talischi and G.H. Paulino},
\textit{On the virtual element method for three-dimensional
linear elasticity problems on arbitrary polyhedral meshes},
Comput. Methods Appl. Mech. Engrg., {\bf 282}, (2014), pp. 132--160.

%\bibitem{GR} 
%\textsc{V. Girault and P.A. Raviart},
%\textit{Finite Element Methods for Navier-Stokes Equations},
%Springer-Verlag, Berlin, 1986.

\bibitem{LBC2012} 
\textsc{Q. Long, P.B. Bornemann and F. Cirak},
\textit{Shear-flexible subdivision shells},
Internat. J. Numer. Methods Engrg., {\bf 90(13)}, (2012) pp. 1549--1577.


\bibitem{Lova} 
\textsc{C. Lovadina},
\textit{A brief overview of plate finite element methods},
Integral methods in science and engineering.
Vol. 2, 261--280, Birkh\"auser Boston, Inc., Boston, MA, 2010.




\bibitem{LNS22} 
\textsc{M. Lyly, J. Niiranen and R. Stenberg},
\textit{A refined error analysis of MITC plate elements},
Math. Models Methods Appl. Sci., {\bf 16}, (2006), pp 967--977.


\bibitem{MRR2015} 
\textsc{D. Mora, G. Rivera and R. Rodr\'iguez},
\textit{A virtual element method for the Steklov eigenvalue problem},
Math. Models Methods Appl. Sci., {\bf 25(8)}, (2015), pp. 1421--1445.
 
\bibitem{PG2015} 
\textsc{G.H. Paulino and A.L. Gain},
\textit{Bridging art and engineering using Escher-based virtual elements},
Struct. Multidiscip. Optim., {\bf 51(4)}, (2015), pp. 867--883. 
 
 

\bibitem{PPR15} 
\textsc{I. Perugia, P. Pietra and A. Russo},
\textit{A plane wave virtual element method for the Helmholtz problem},
ESAIM Math. Model. Numer. Anal., {\bf 50(3)}, (2016), pp. 783--808.

\bibitem{RW} 
\textsc{S. Rjasanow and S. Wei\ss er},
\textit{Higher order BEM-based FEM on polygonal meshes},
SIAM J. Numer. Anal., {\bf 50(5)}, (2012), pp. 2357--2378.


\bibitem{scott-zhang}
\textsc{L.R. Scott and S. Zhang},
\textit{Finite element interpolation of nonsmooth functions satisfying boundary conditions},
Math. Comp. {\bf 54}, (1990), pp. 483--493.

\bibitem{Stein}
\textsc{E. M. Stein}, 
\textit{Singular integrals and differentiability properties of functions},
volume 2. Princeton University Press, 1970.


\bibitem{ST04} 
\textsc{N. Sukumar and A. Tabarraei},
\textit{Conforming polygonal finite elements},
Internat. J. Numer. Methods Engrg., {\bf 61}, (2004), pp. 2045--2066.

\bibitem{TPPM10}
\textsc{C. Talischi, G.H. Paulino, A. Pereira and I.F.M. Menezes},
\textit{Polygonal finite elements for topology optimization: A unifying paradigm},
Internat. J. Numer. Methods Engrg., {\bf 82(6)}, (2010), pp. 671--698.

\bibitem{WRR}
\textsc{P. Wriggers, W.T. Rust and B.D. Reddy},
\textit{A virtual element method for contact},
Comput. Mech., {\bf 58}, (2016), pp. 1039--1050.

\bibitem{ZChZ2016}
\textsc{J. Zhao, S. Chen and B. Zhang},
\textit{The nonconforming virtual element method for plate bending problems},
Math. Models Methods Appl. Sci., {\bf 26(9)}, (2016), pp. 1671--1687.



\end{thebibliography}

\end{document}